\documentclass{amsart}

\usepackage{amssymb}
\usepackage{amscd}
\usepackage{amsmath}
\usepackage{enumerate}
\usepackage{tensor}
\usepackage{mathtools}
\usepackage[all,cmtip, arrow, matrix, curve]{xy}
\usepackage{tikz-cd}
\usepackage{tikz-3dplot}
\usepackage{graphicx}
  \usepackage{wrapfig}
  \usetikzlibrary{decorations.pathreplacing,backgrounds, matrix}
  \usetikzlibrary{automata, patterns, calc}

\usepackage{afterpage}
\usepackage{setspace}

\usepackage{graphicx,import} 
\usepackage{caption}
\usepackage{subcaption}
\usepackage{color}
\usepackage{transparent}
\usepackage{hyperref}
\usepackage[normalem]{ulem}

\usepackage{color}

\newtheorem{theo}{Theorem}
\newtheorem{lema}[theo]{Lemma}
\newtheorem{cor}[theo]{Corollary}
\newtheorem{prop}[theo]{Proposition}

\newtheorem{definition}[theo]{Definition}

\newtheorem{remark}[theo]{Remark}
\newtheorem{notation}[theo]{Notation}

\newtheorem{example}[theo]{Example}
\newtheorem{problem}[theo]{Problem}

\newcommand{\BB}{{\mathbb{B}}}

\newcommand{\NN}{{\mathbb{N}}}

\newcommand{\QQ}{{\mathbb{Q}}}
\newcommand{\RR}{{\mathbb{R}}}

\newcommand{\SSS}{{\mathbb{S}}}
\newcommand{\ZZ}{{\mathbb{Z}}}

\newcommand{\calA}{{\mathcal{A}}}

\newcommand{\calG}{{\mathcal{G}}}
\newcommand{\calH}{{\mathcal{H}}}

\newcommand{\calN}{{\mathcal{N}}}

\newcommand{\calP}{{\mathcal{P}}}
\newcommand{\calS}{{\mathcal{S}}}

\newcommand{\calU}{{\mathcal{U}}}

\newcommand{\p}{\frak{p}}
\newcommand{\q}{\frak{q}}
\newcommand{\fP}{\frak{P}}
\newcommand{\fh}{\frak{h}}

\newcommand{\comp}{{\circ}}

\begin{document}

\title[]{Moderately Discontinuous Homotopy}

\author{J. Fern\'andez de Bobadilla}
\address{Javier Fern\'andez de Bobadilla:  
(1) IKERBASQUE, Basque Foundation for Science, Maria Diaz de Haro 3, 48013, 
    Bilbao, Bizkaia, Spain
(2) BCAM  Basque Center for Applied Mathematics, Mazarredo 14, E48009 Bilbao, 
Basque Country, Spain 
(3) Academic Colaborator at UPV/EHU} 
\email{jbobadilla@bcamath.org}

 \author{S. Heinze}
 \address{Sonja Heinze:  
BCAM  Basque Center for Applied Mathematics, Mazarredo 14, E48009 Bilbao, 
Basque Country, Spain }
\email{sheinze@bcamath.org}

 \author{M. Pe Pereira}
 \address{Facultad de Ciencias Matematicas  and Instituto de Matem\'atica Interdisciplinar, Universidad Complutense de Madrid, Plaza de Ciencias, 3,  Ciudad Universitaria, 28040 MADRID} 
\email{maria.pe@mat.ucm.es}

\thanks{The first author is supported by ERCEA 615655 NMST Consolidator Grant, MINECO by the project 
reference MTM2016-76868-C2-1-P (UCM), by the Basque Government through the BERC 2018-2021 program and Gobierno Vasco Grant IT1094-16, by the Spanish Ministry of Science, Innovation and Universities: BCAM Severo Ochoa accreditation SEV-2017-0718. The second author is supported by a La Caixa Ph.D. grant associated to the MINECO project SEV-2011-0087, by ERCEA 615655 NMST Consolidator Grant, MINECO by the project reference MTM2016-76868-C2-1-P (UCM), by the Basque Government through the BERC 2018-2021 program and Gobierno Vasco Grant IT1094-16, by the Spanish Ministry of Science, Innovation and Universities: BCAM Severo Ochoa accreditation SEV-2017-0718. 
The third author is supported by MINECO by the project reference MTM 2017-89420 and MTM2016-76868-C2-1-P (UCM).}

\begin{abstract}
We introduce a metric homotopy theory, which we call Moderately Discontinuous Homotopy, designed to capture Lipschitz properties of metric singular subanalytic germs. It matches with the Moderately Discontinuous Homology theory receantly developed by the authors and E. Sampaio. The $k$-th MD homotopy group is a group $MDH^b_\bullet$ for any $b\in [1,\infty]$ together with homomorphisms $MD\pi^b\to MD\pi^{b'}$ for any $b\geq b'$. We develop all its basic properties including finite presentation of the groups, long homology sequences of pairs, metric homotopy invariance, Seifert- van Kampen Theorem and the Hurewicz isomorphism Theorem. We prove comparison theorems that allow to relate the metric homotopy groups with topological homotopy groups of associated spaces. For $b=1$ it recovers the homotopy groups of the tangent cone for the outer metric and of the Gromov tangent cone for the inner one. In general, for $b=\infty$ the $MD$- homotopy recovers the homotopy of the punctured germ. Hence, our invariant can be seen as an algebraic invariant interpolating the homotopy from the germ to its tangent cone. We end the paper with a full computation of our invariant for any normal surface singularity for the inner metric. We also provide a full computation of the MD-Homology in the same case.
\end{abstract}

\maketitle
\tableofcontents

\section{Introduction}
We introduce a metric homotopy theory, which we call Moderately Discontinuous Homotopy, designed to capture Lipschitz properties of metric singular subanalytic germs.
This theory matches with the Moderately Discontinuous Homology theory receantly developed by the authors and E. Sampaio in \cite{MDH}. 
Both theories run parallel to the classical theories and are related through a theorem of type Hurewicz from the Moderately Discontinuous Homotopy groups to the MD Homology ones. 

With this metric algebraic topology we aim to provide algebraic and numerical invariants which capture Lipschitz phenomena of real and complex analytic singularities, and more generally, subanalytic sets, categories that we hope to  enlarge in the future. 

The object of study are pointed subanalytic germs, which have conical structure, endowed with a metric that has to be equivalent to a subanalytic one (such as the outer or the inner). A base point in a germ $(X,0)$ is what we call a point in these theories, that is a mapping from the interval $([0,1],0)$ to the germ $(X,O)$ preserving up to a constant the distance to the origin.

An $(n,b)$-loop in a metric germ $(X,0,d)$ is given by a possibly discontinuous map from the cone of the n-cube $I^n$ to $(X,O)$ where we allow discontinuities that measured with respect to the metric $d$ are no bigger than $t^b$ where $t$ is the distance to the origin of the germ. This type of mappings are new and we call them weak $b$-maps (see Definition \ref{def: weak b-maps}).

Then, for every $b\in (0,+\infty)$ we consider the set of $(n,b)$-loops up to weak $b$-homotopies (see definitions  \ref{def:weakhomotopy}), which admit also discontinuities up to order $t^b$. This set with the operation of concatenation is a group that we denote by $MD\pi_n^b(X,d)$, which is abelian for $n\geq 2$.  Moreover there are homomorphisms relating the $(n,b)$-MD Homotopy group with the $(n,b')$-MD Homotopy group.  The $b$-MD homotopy groups of a metric germ $(X,O,d)$ capture the homotopy nature of the germ up to gaps that measured in the metric $d$ are of size smaller than $t^b$.  

In the case of the outer and inner metric, we prove that all the groups $MD\pi_n^b$  are finitely generated  abelian groups. 
Moreover, only finitely many homomorphisms $MD\pi^b_\bullet\to MD\pi^{b'}_\bullet$ are essential. 

The MD homotopy groups  are a bi-Lipschitz subanalytic invariant. They are functorial from the category of metric subanalytic germs, both with lipschitz maps and $b$-maps (which are a kind of piecewise continuous mappings that were introduced in \cite{MDH}, see Definition \ref{def:bmaps}) as morphism. These groups are also invariant by suitable metric homotopies, and satisfy versions of the relative long exact sequence
 and they are independent of the base point for a $b$-connected germ (in the sense of \cite{MDH}, see Definition \ref{def:b-connected}). 
 
We prove a Seifert van Kampen type theorem for the $b$-MD fundamental groupoid for coverings good enough. That is, we get conditions so that it is the colimit of the $b$-MD fundamental groupoids of the elements of the covering as in the classical theorem (see Theorem \ref{theo:SvKGroupoids}). 
Seifert van Kampen Theorem is one of the delicate aspects of our theory. This is no surprise, since Mayer-Vietoris for $MD$-Homology was also delicate to formulate and prove. One subtlety lies in finding the appropriate notion of covering for which the result may hold. There are two conditions ($(*)_b$ and $(**)_b$) that have to be satisfied in order that Theorem \ref{theo:SvKGroupoids} is satisfied. Before we prove it we show a more general theorem  (see Theorem \ref{theo:1stuniversal}) that is satisfied only assuming condition $(*)_b$.

We prove comparison theorems, which relate the MD-homotopy groups of a germ with the outer metric, with usual homotopy group of limits of $b$-horn neighbourhood of it (see Theorem \ref{theo:maincomparacion}). This result has an analogous in the MD Homology theory that was conjetured by Lev Birbrair and proved in \cite{MDH}. Here it is not possible to adapt the proof of the corresponding result for homology, since it is based in Mayer-Vietoris sequences, that are not available in homotopy. Instead, we need to perform an interpolation procedure that "metrically homotopes" a discontinuous weak $b$-map to a continuous one (see the proof of Proposition \ref{prop:directlimit3} and its preparations for details.).

The comparison theorems have many important consequences for pointed pairs of metric germs with the inner or the outer metric: the Hurewicz isomorphims theorem is satisfied (see Theorem \ref{theo:hur} and its proof at the end of  Section \ref{sec:hurew}), the MD homotopy groups are finitely presented, given a pointed pair of metric subanalytic sets, the set of $b$'s such that the groups $MD\pi^{b+\epsilon}_k$ and $MD\pi^{b-\epsilon}_k$ differ for sufficiently small positive $\epsilon$ is finite and contained in $\QQ$, for $b=1$ and the outer metric $MD\pi^{1}_k$ recovers the $k$-th homotopy group of the punctured tangent cone, for the inner metric $MD\pi^{1}_k$ recovers the $k$-th homotopy group of the punctured Gromov tangent cone, and $MD\pi^{\infty}_k$ recovers the $k$-th homotopy group of the link (see Proposition~\ref{prop:MD homotopy at infinity} and Corollary \ref{cor:finitudes}). The $b=\infty$ comparison theorem (Proposition~\ref{prop:MD homotopy at infinity}) is easier and has a direct proof. Let us also comment that for the proof of this statements for the inner metric we need to use an adequate re-embedding of the germ that reduces the assertions to the outer metric case.

As a first example we compute the $MD$-homotopy for the $b$-cones, completing the computation also for the $MD$-homology started in \cite{MDH}. 

The inner metric of normal complex surface singularities is fully described in \cite{BirbrairNeumannPichon:2014}. In the last section we use this description and the metric version of the Seifert van Kampen theorem, Theorem \ref{theo:SvKGroupoids}, to compute all the $b$-MD fundamental groups of a complex singularity surface germ with the inner metric. Given a surface germ $(X,0)$ we give a $b$-\emph{homotopy model} $X^b_\epsilon$ which is a topological space whose homology and fundamental group is the $b$-MD Homology and $b$-MD fundamental group of $(X,0,d_{inn})$ (see Theorems \ref{theo:MDpi1inner}-\ref{theo:MDHSurfaces}). The space $X^b_\epsilon$ has the homotopy type of a plumbed $3$-manifold in which several circles are identified (a ``branched $3$-manifold'' in the language of~\cite{BirbrairNeumannPichon:2014}). Such a branched $3$-manifold has a natural description compatible with a JSJ decomposition of the link $X_\epsilon$.

It is well known that the fundamental group of the link of a normal surface singularity determines the topology of the singularity, except in the cyclic quotient case. This is a particular case of a theorem of Waldhausen. We end the paper with a list of open questions (See Problems~\ref{prob:1},~\ref{prob:2}), out of which the following two stand out:
\begin{itemize}
 \item Does the MD fundamental group of a normal complex surface singularity determine the inner geometry of the surface, in the non cyclic quotient case?
 \item Can the Lipschitz normally embedded property of a normal surface singularity be read from all the $b$-MD fundamental groups for the inner and the outer metric?
 
\end{itemize}

\section{Setting and notation}\label{sec:notation}
We recall the basic definitions that were introduced in \cite{MDH}, where the reader can find a more detailed exposition. We will always work with bounded subanalytic subsets, which in particular are globally subanalytic (see \cite{Dries:1986}). Recall that the collection of all globally subanalytic sets forms an O-minimal structure (see \cite{Dries:1986}). We use \cite{Dries:1998} and \cite{Coste:1999} as basic references in O-minimal geometry.

\begin{definition}\label{def:germ}
A \emph{metric subanalytic germ} $(X,x_0,d)$  is a subanalytic germ $(X,x_0)$  such that $x_0\in \overline{X}$ (where $\overline{X}$ denotes the closure of $X$ in $\RR^m$) and $d$ is a subanalytic metric that induces the same topology on $X$ as the restriction of the standard topology on $\RR^n$ does . 
We omit $x_0$ and $d$ in the notation when it is clear from the context. We say $x_0$ is the vertex of the germ.

Given two  germs $(X,x_0)$ and $(Y,y_0)$, a \emph{subanalytic map germ} $f:(X,x_0)\to (Y,y_0)$ is a subanalytic continuous map $f:X\to Y$ that admits a continuous and subanalytic extension to a map germ  $\overline{f}:(X\cup\{x_0\},x_0)\to (Y\cup\{y_0\},y_0)$. 
\end{definition}
   
We define in the expected manner metric subanaytic subgerms, pairs of metric subanalytic germs and subanalytic mappings between them,  etc. For more details, see section 2.2 in \cite{MDH}.

We recall that, as in \cite{MDH}, it is possible for a germ $(X,x_0)$,  according to our definition,  that $x_0\notin X$. Important examples of germs of these types are open subgerms $(U,0)$ of a subanalytic germ $(X,0)$. 

Also, notice that we are building a bilipschitz invariant, so, in practice metrics that are not subanalytic but are bilipchitz equivalent to subanalytic ones are allowed; for example the inner metric.  See \cite{MDH} for more details.

\begin{definition}
\label{def:lva}

A map germ  $f: (X, x_0) \rightarrow (Y, y_0)$ is said to be \emph{linearly vertex approaching} (l.v.a. for brevity) if there exists $K\geq 1$ such that $$\frac{1}{K}||x-x_0||\leq ||f(x)-y_0||\leq K||x-x_0||$$ for every $x$ in some representative of $(X,x_0)$. {The constant $K$ is called the l.v.a constant for $f$}.

Let $(X, x_0, d_X)$ and $(Y, y_0, d_Y)$ be two metric subanalytic germs.
 A \emph{Lipschitz linearly vertex approaching subanalytic map germ} $f: (X, x_0,d_X) \rightarrow (Y, y_0,d_Y)$ is a map germ that is both Lipschitz (with respect to the metrics $d_X$ and $d_Y$) and l.v.a.. In particular one can take the same constant $K$ for both properties. 
\end{definition}

Following the philosophy of~\cite{MDH}, the metric germ $((0,\epsilon),0,d)$, that is the interval germ  with the euclidean metric, plays the role of a point. This motivates the following definition, which fixes the category of pointed pairs of metric subanalytic germs (the source category of the Lipschitz homotopy functors introduced in this paper).

\begin{definition} 
 The \emph{category of pairs of metric subanalytic germs} has pairs of metric subanalytic germs as objects and Lipschitz l.v.a. subanalytic maps of pairs as morphisms.  
 
A {\em point in $X$} where $(X,x_0,d)$ is a metric subanalytic germ,  is a continuous l.v.a. subanalytic map germ $\frak{p}:(0,\epsilon) \to X$. A {\em pointed pair of metric subanalytic germs} is a pair $(X,Y,x_0,d,\frak{p})$ of metric subanalytic germs together with a point $\p$ in $Y$. Given pointed metric subanalytic germs $(Y,x_0,d,\frak{p})$ and $(Y',x'_0,d',\p')$, a morphism $f:(Y,x_0,d)\to (Y',x'_0,d')$ preserves the base point if $f\comp \p=\p'$. The \emph{category of pointed pairs of metric subanalytic germs} has pointed pairs of metric subanalytic germs as objects and Lipschitz l.v.a. subanalytic maps of pairs that preserves the base point as morphisms. 
 \end{definition}

As in~\cite{MDH} we can enlarge the category admitting morphisms that are $b$-maps, see Definition \ref{def:bmaps}.

The following definition is important: 

\begin{definition}\label{def:bpoint}[$b$-point]
Let $(X,x_0,d)$ be a metric subanalytic germ and let $b\in (0, +\infty)$. 
Two points $\p$ and $\q$ in $X$ are  {\em $b$-equivalent}, and we write $\p \sim_b \q$, if 
$$
\lim_{t \to 0} \frac{d_X(\p(t), \q(t))}{t^b} = 0.
$$

An equivalence class of points is called a \emph{$b$-point} of X. 
\end{definition}

\subsection{Useful terminology}
\label{sec:terminology}
In the development of metric homotopy theory we will use $n$-loops and homotopies modeled on cones over cubes. We introduce some related notation here: 
\begin{notation}\label{not:cone}
In this chapter $I := [0,1]$ denotes the unit interval.

We consider the cone over the cube $I^n$:
$$C(I^n)=\{(tx,t)\in \RR^{n}\times \RR;\, x\in I^n, \ t\in (0,+\infty )\}.
$$


For convenience the origin is not in $C(I^n)$. 

Note that the notion of l.v.a. maps from or into $C(I^n)$ can be thought using the parameter $t$ instead of the distance to the origin. 

We sometimes denote $(y_1,...,y_n) \in I^n$ by $y_{1 .. n}$ or also by $ (y_{1 .. n-1}, y_n)$ and similarly.

In a similar way we denote by $C(M)$ the  cone over a bounded subanalytic set $M\subset\RR^n$.

For readability, we write $(y, t)$ to denote $(yt, t)$ in $C(M)$.
 But be aware that this does only provide a system of coordinates out of the origin of $C(M)$.
\end{notation}

\begin{definition}[Normal point in $C(I^n)$]\label{def:normal point in C(I^n)}

Let $\q: (0, \epsilon) \to C(I^n)$ be a point in $C(I^n)$. 
We say that $\q$ is a \emph{normal point} if $\q(t)$ is expressed as $(\alpha(t), t)$ in the usual coordinates of $C(I^n)$. 

\end{definition} 

\begin{notation}\label{not:Bcategory} We respectively denote  by $\calS$,   $\calG$,  and  $\calG\calP$ the categories of sets, groups, and groupoids. 


We denote by $\BB$ the category  whose set of objects is $(0,\infty]$ and there is a unique morphism from $b$ to $b'$ if and only if $b\geq b'$.

The {\em category $\BB-\calS$, (resp. $\BB-\calG$, $\BB-\calG\calP$) of $\BB$-sets (resp. $\BB$-groups, $\BB$-groupoids}) is the category whose objects are functors from $\BB$ to $\calS$  (resp. $\calG$, $\calG\calP$) and the morphisms are natural transformations of functors.
\end{notation}

\subsection{Conical structures}

Given a subanalytic germ $(X,x_0)$, its link $C_X$ is the intersection of $X$ with a small sphere centered in $x_0$. Its subanalytic homeomorphism type is independent of the radius. 
\begin{definition}\label{def:con_struc} Given a subanalytic germ $(X,x_0)$ and a family of subanalytic subgerms $(Z_1,0)$,...,$(Z_k,0)\subseteq(X,0)$, a conical structure for $(X,0)$ compatible with the the family $\{Z_i \}$ is a subanalytic homeomorphism $h:  C(L_X)\to (X,x_0)$ such that $||x_0-h(x,t)||=t$ and such that $h(C(L_{Z_i}))=Z_i$ with $L_{Z_i}$ in $L_X$. 
\end{definition}

Conical structures always exist (see \cite{Coste:1999} Theorem 4.10). The following remark, proved in~\cite{MDH} will be used.

\begin{remark} 
\label{rem:lvareduction}
Let $(X,x_0)$ be a subanalytic germ with compact link. Consider any subanalytic map germ $f: (X, x_0) \rightarrow (Y, y_0)$ that is a homeomorphism onto its image. Let $\{Z_j\}_{j\in J}$ be a finite collection of closed subanalytic subsets of $X$. There is a subanalytic homeomorphism germ $\phi:(X,x_0)\to (X,x_0)$ such that $\phi(Z_j)=Z_j$ for all $j\in J$ and such that $||f\comp \phi(x)-y_0||=||x-x_0||$ (this is stronger than l.v.a.).
\end{remark}

\section{$b$-moderately discontinuous homotopy groups and their basic properties}\label{sec:def}

\subsection{Weak $b$-maps}

Weak $b$-maps are a way of weakening the continuity of loops and homotopies in the classical theory, in order to establish a parallel theory that captures metric phenomena. 

\begin{definition}[Weak $b$-map]\label{def: weak b-maps}
Let $(X, x_0, d_X)$ be a metric subanalytic germ.  
Let $b \in (0, \infty)$. A {\em weak $b$-moderately discontinuous subanalytic map (weak $b$-map, for abbreviation) from a subanalytic germ $(Z,0)$ to $(X, x_0, d)$} is a finite collection $\{(C_j,f_j)\}_{j\in J}$, where $\{C_j\}_{j\in J}$ is a finite closed subanalytic cover of $(Z, \underline{0})$ and $f_j: C_j \to X$ are continuous l.v.a. subanalytic maps for which for any point $\q$ in $C_{j_1} \cap C_{j_2}$  for any $j_1, j_2 \in J$ we have that $f_{j_1} \comp \q\sim_b f_{j_2}\comp \q$. We call $\{C_j\}_{j\in J}$ the {\em cover of the weak $b$-map $\{(C_j,f_j)\}_{j\in J}$}.
%

Two weak $b$-maps $\{(C_j,f_j)\}_{j\in J}$ and $\{(C'_k,f'_k)\}_{k\in K}$ from $(Z,0)$ to $(X,x_0,d)$ are called {\em $b$-equivalent}, if for any point $\q$ contained in the intersection $C_j \cap C'_k$ for any $j \in J$ and $k \in K$, we have that  $f_j\circ \q \sim_b f'_k \comp \q$.

We make an abuse of language and we also say that a weak $b$-map from $(Z,z_0)$ to $(X,x_0,d_X)$ is an equivalence class as above. 

For $b = \infty$, a weak $b$-map from $Z$ to $X$ is a continuous l.v.a. subanalytic map germ from $(Z,z_0)$ to $(X,x_0,d_X)$.
\end{definition}
Informally we can say that a weak $b$-map gives a well defined mapping from $Z$ to the set of $b$-points $X$ (see Definition \ref{def:bpoint}). 

\begin{remark}[Gluing of weak $b$-maps]\label{rem:glueweak}
Two weak $b$-maps $\varphi_1$ and $\varphi_2$ defined on $Z_1$ and $Z_2$ respectively glue to a  weak $b$-map defined on $Z_1\cup Z_2$ if and only if for any point $\q$ in $Z_1 \cap Z_2$ we have the equivalence $\varphi_1 \comp \q\sim_b\varphi_2 \comp \q$. 
\end{remark}

\begin{remark}[Equivalence by refinement]\label{rem:refineweak}
Let $\varphi = \{(C_j, f_j)\}_{j \in J}$ be a weak $b$-map  and $\{D_k\}_{k \in K}$ a refinement of $\{C_j\}_{j\in J}$. For $k \in K$, let $r(k) \in J$ be such that $D_k \subseteq C_{r(k)}$. Then $\{(D_k, f_{r(k)}|_{D_k})\}_{k \in K}$ is equivalent to $\varphi$. As a consequence any weak $b$-map from $Z$ to $X$ has a representative $\{(C_j,f_j)\}_{j\in J}$, for which the interior of $C_{j_1} \cap C_{j_2}$ is empty for any $j_1, j_2 \in J$.

\end{remark}


\begin{remark}\label{rem:b'weak}
If  $b\geq b'$ then any weak $b$-map is also a weak $b'$-map. 
\end{remark}

We recall the  definition of $b$-maps, with respect to which the $b$-MD-Homology groups are functorial (see Section 5 in \cite{MDH}). We will also prove functoriality of the MD-Homotopy groups for this type of morphisms. In particular, we can compose a weak $b$-map with a $b$-map on the right, as we see in Definition \ref{def:comp_weak} below.

\begin{definition}[Category pointed of metric pairs with $b$-maps]
\label{def:bmaps}
Let $(X, x_0, d_X)$ and $(Y, y_0, d_Y)$ be metric subanalytic germs, $b \in (0, +\infty)$. A \em $b$-moderately discontinuous subanalytic map (a $b$-map, for abbreviation) from $(X, x_0, d_X)$ to  $(Y, y_0, d_Y)$ is a finite collection $\{(C_i,f_i)\}_{i\in I}$, where $\{C_i\}_{i\in I}$ is a finite closed subanalytic  cover of $X$  and the $f_i: C_i \to Y$ are l.v.a. subanalytic maps satisfying the following: for any $b$-equivalent pair of points $\p$ and $\q$ contained in $C_i$ and $C_j$ respectively ($i$ and $j$ may be equal), the points $f_i\circ \p$ and $f_j\circ \q $ are b-equivalent in $Y$.

Two $b$-maps $\{(C_i,f_i)\}_{i\in I}$ and $\{(C'_i,f'_i)\}_{i\in I'}$ are called {\em $b$-equivalent} if for any $b$-equivalent pair of points $\p$, $\q$ with Im$(\p) \subseteq C_i$ and Im$(\q) \subseteq C'_{i'}$, the points $f_i\circ \p$ and $f'_{i'}\circ \q$ are b-equivalent in $Y$.

We make an abuse of language and we also say that a $b$-map from $(X,x_0,d_X)$ to $(Y,y_0,d_Y)$ is an equivalence class as above. 

For $b = \infty$, a $b$-map from $X$ to $Y$ is a l.v.a. subanalytic map from $X$ to $Y$.

A $b$-map between pointed pairs of metric subanalytic germs $(X,Y,x_0,d, \p)$ and $(\widetilde{X}, \widetilde{Y},\widetilde{x}_0, \widetilde{d}, \widetilde{\p})$ is a $b$-map from $X$ to $\widetilde{X}$ admitting a representative $\{(C_i, f_i )\}_{i \in I}$  for which 
\begin{enumerate}
 \item for any point $\q$ of $Y\cap C_i$, the point $f_i\comp \p$ is $b$-equivalent to a point in $\widetilde{Y}$,
 \item if the image of $\p$ is in $C_i$ then $f_i\comp \p\sim_b\widetilde{\p}$.
\end{enumerate}
 
The {\em category of pointed metric pairs with $b$-maps} has as objects pointed pairs of metric subanalytic germs and as morphisms the $b$-maps between them.
\end{definition}

Informally, we can say that $b$-maps give well defined mappings from the set of $b$-points of $X$ to the set of $b$-points of $Y$.

Note that the analogue for $b$-maps of Remarks \ref{rem:glueweak} and \ref{rem:b'weak} for weak $b$-maps are not satisifed (unless the target of the $b$-map is a convex set; see \cite{thesis} for more details).

\begin{definition}[Composition of weak $b$-maps and $b$-maps]\label{def:comp_weak}
Let $Z$ and $Z'$ be subanalytic germs and let $(X,x_0,d)$ and $(X',x_0,d')$ be metric subanalytic germs.  Let $\varphi = \{(C_j, f_j)\}_{j \in J}$ be a weak $b$-map from $Z$ to $X$. 

For a continuous l.v.a. subanalytic map $\phi$ from $Z'$ to $Z$, we can define $\varphi \comp \phi$ to be the weak $b$-map $\{(\phi^{-1}(C_j),\varphi_j \comp \phi )\}_{j \in J}$ from $Z'$ to $X$. 

Let $\psi = (D_k, g_k)_{k \in K}$ be a $b$-map from $X$ to $X'$. We define $\psi \comp \varphi$ to be the weak $b$-map $\{ (f_j^{-1}(D_k) \cap C_j, g_k \circ f_{j| f_j^{-1}(D_k) \cap C_j})\}_{(j,k) \in J \times K}$ from $Z$ to $X'$. 
\end{definition}

\subsection{Definition of the $b$-moderately discontinuous metric homotopy groups}
\label{subsection:defMDhomology}

\begin{definition}[$(n,b)$-loop]\label{def:bn-loop}
Let $(X, Y, x_0, d, \p)$ be a pointed pair of metric subanalytic germs,  $n\in\mathbb{N}$ and $b \in (0, \infty]$. A {\em $b$-moderately discontinuous $n$-loop (a ($(n,b)$-loop, for short)} is a weak $b$-map $\varphi$ from $C(I^n)$ to $(X,Y)$ for which the following boundary conditions hold: 
\begin{enumerate}[(a)]
 \item for any point $\q$ in $C(\partial I^n)$, the point $\varphi \comp \q$ is $b$-equivalent to a point in $Y$.
 \item for any normal point $\q$ in $C(\partial I^n\setminus (I^{n-1}\times\{1\}))$, we have $\varphi \comp \q\sim_b\p$.
\end{enumerate}

We denote the set of all $(n,b)$-loops in $(X, Y, \p)$ by $MD\Gamma_n^b(X,Y,\p)$. Observe that we suppress $x_0$ and $d$ in the notation $MD\Gamma_n^b(X,Y,\p)$, even though they influence the set of $(n,b)$-loops in $(X, Y, \p)$.
In the case $Y$ coincides with the image of $\p$, we simply write $MD\Gamma_n^b(X,\p)$
\end{definition}

\begin{notation}\label{not:maps} Denote the inclusions 
$\iota_s:C(I^n)\to C(I^{n+1})$ defined as $\iota_s(y,t):=((y,s),t)$ for any $s\in [0,1]$, and the projection $\rho:C(I^{n+1}) \to C(I^n)$ defined by $\rho(y_{1..n+1},t):= (y_{1..n}, t)$. 
\end{notation}

\begin{definition}[Weak $b$-homotopy (relative to $W$)]\label{def:weakhomotopy}
Let $(X, x_0, d)$ be a metric subanalytic germ. Let $\varphi_0$, $\varphi_1$ be weak $b$-maps from $C(I^n)$ to $X$ . A {\em weak $b$-homotopy from  $\varphi_0$ to $\varphi_1$} is a weak $b$-map $H$ from $C(I^{n+1})$ to $X$ such that $H\circ \iota_0=\varphi_0$ and $H\circ \iota_1=\varphi_1$ where $\iota_s$ denotes the inclusion of $C(I^n)$ into $C(I^{n+1})$ given by $(y,t) \to ((y,s),t)$. 

We say $H$ is a weak $b$-homotopy \emph{relative to a subgerm}  $(W, \underline{0}) \subseteq (C(I^n), \underline{0})$ if  moreover 
$$H\comp \q\sim_b\varphi_0 \comp \rho  \comp  \q$$
for any point $\q$ in $\rho^{-1}(W)$.


In case $\varphi_0$, $\varphi_1$ 
satisfy that for any point $\p$ in 
 a subanalytic germ $K$ of $C(I^n)$ the points $\varphi_0\circ \p$ and $\varphi_1\circ \p$ are  $b$-equivalent to  points in certain subgerm $(Y,x_0)$ of $(X,x_0)$, and for any point $\q$ in $\rho^{-1}(K)$  the point $H \comp \q$ is $b$-equivalent to a point in $Y$ 
then we say that 
$H$ \em{preserves the inclusion of $K$ in $Y$}. 

 
We say that $\varphi_0$ and $\varphi_1$ are {\em weakly $b$-homotopically equivalent} or \em{weak $b$-homotopic} (\em{relative to $W$} or preserving the inclusion of $K$ in $Y$ if it applies). 
 \end{definition}
 


\begin{figure}\label{fig:weak_b_homo}
\includegraphics[scale=0.15]{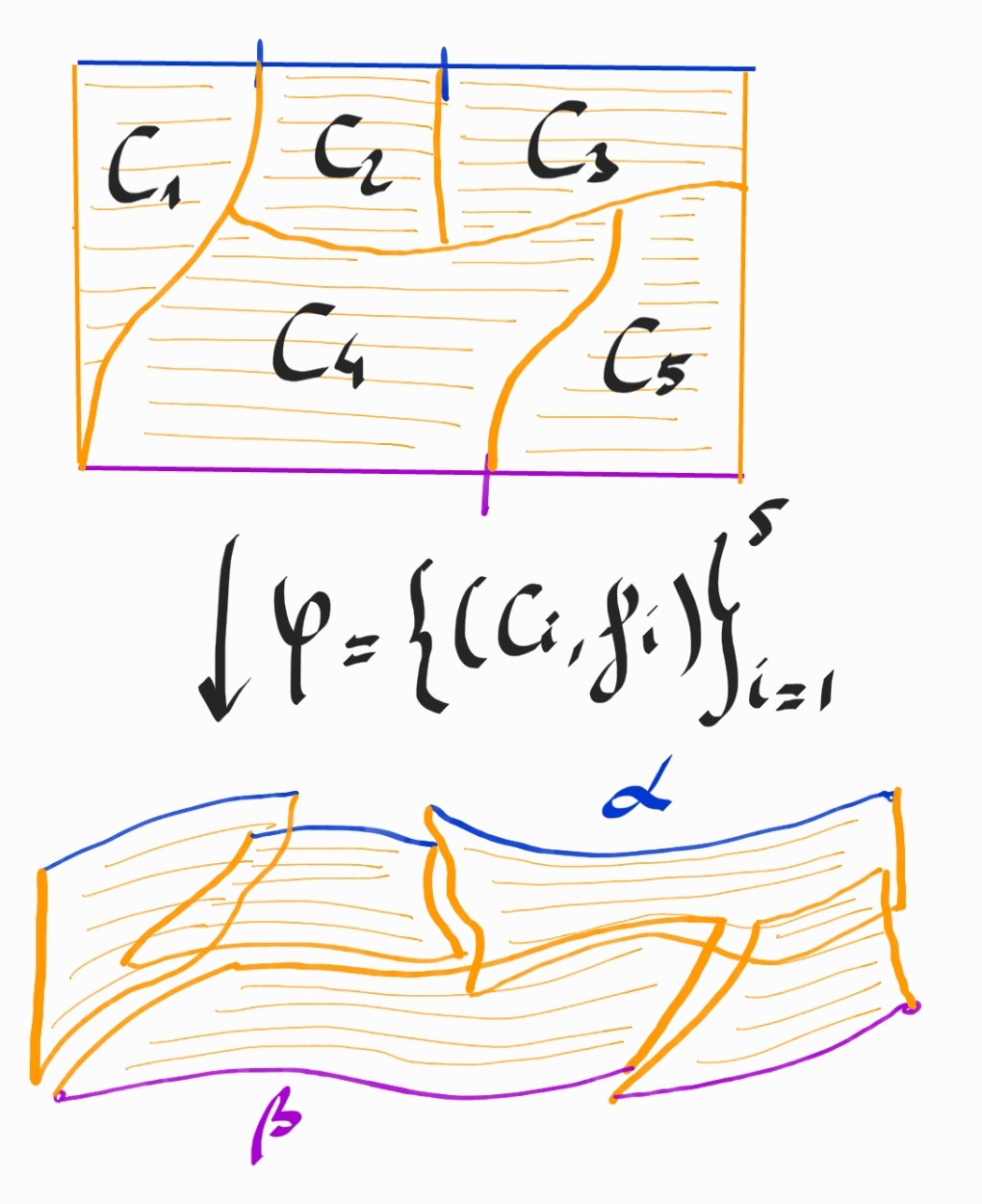}
\caption{Squematic example of a weak $b$-homotopy joining the weak $b$-maps $\alpha$ and $\beta$ in blue and purple. }
\end{figure}


\begin{remark}\label{rem:loopsbound}
Let $\varphi_0$ and $\varphi_1$ be weak $b$-maps from $C(I^n)$ to $X$. Let $(W, 0) \subseteq (C(I^n), 0)$ be a subgerm. There is a necessary condition for $\varphi_0$ and $\varphi_1$ to admit a weak $b$-homotopy relative to $W$ between them: for any point $\q$ in $W$,  the points $\varphi_0 \comp \q$ and $\varphi_1 \comp \q$ are $b$-equivalent.

In particular any two $(n,b)$-loops $\varphi_1$ and $\varphi_2$ in $(X, Y, \p)$ fulfill the necessary condition to admit a weak $b$-homotopy relative to $ W = C(\partial I^n\setminus I^{n-1}\times\{1\})$ between them. Moreover, if $H$ is such weak $b$-homotopy, then $H \comp \iota_s = \varphi_s$ is a $b$-moderately discontinuous $n$-loop for any $s\in [0,1]$.
\end{remark}
\begin{proof}
We prove the last statement. Let $\q$ be a point in $C(\partial I^n\setminus I^{n-1}\times\{1\})$. By Remark~\ref{rem:lvareduction}, there is a subanalytic homeomorphism $h:(0, \epsilon) \to (0, \epsilon)$, for which $\widetilde{\q} := \q \comp h$ is a normal point. Since both $\q$ and $\widetilde{\q}$ are l.v.a., the homeomorphism $h$ is also l.v.a.. Therefore, the $b$-equivalence between $\varphi_1 \comp \widetilde{q}$ and $\varphi_2 \comp \widetilde{q}$ implies the $b$-equivalence between $\varphi_1 \comp \q$ and $\varphi_2 \comp \q$.
\end{proof}

Let us see some easy ways of getting weak $b$-homotopies between weak $b$-maps: 


\begin{example}\label{rem:topological isotopy}
Let $\eta : C(I^n) \times I \to C(I^n)$ be a subanalytic continuous  homotopy with  $\eta_0 = id_{C(I^n)}$ and satisfying that 
there is a $K\geq 1$ such that for any $s\in I$, $\eta_s:=\eta(\_ ,s)$ is l.v.a. for the constant  $K$. 
We define $\hat{\eta} : C(I^{n+1}) \to C(I^n)$ by the formula $\hat{\eta}(y_{1..n+1}, t) := \eta((y_{1..n}, t), y_{n+1})$. Then $\varphi \comp \hat{\eta}$ defines a weak $b$-homotopy from $\varphi$ to $\varphi \comp \eta_1$.
\end{example}

We are going to define concatenations of weak $b$-maps and equip $MD\pi_n^b(X,\p)$ with a product operation, endowing it with a group structure in most cases.

A concatenation of $(n,b)$-loops $\varphi_0$ and $\varphi_1$ is defined, similarly to the classical case, by gluing them along the faces $C(\{1\}\times I^{n-1})$ and $C(\{0\}\times I^{n-1})$. 

We will  define concatenation for weak $b$-maps that are not necessarily $(n,b)$-loops. The following auxiliary mappings will be used: 
\begin{notation}

Let $n\in \mathbb{N}$ and let $0 \leq \alpha_1 \leq \alpha_2 \leq 1$ and $0 \leq \alpha'_1 < \alpha'_2 \leq 1$. Then, $\phi_{\alpha_1, \alpha_2}^{\alpha'_1, \alpha'_2}$ denotes the continuous subanalytic l.v.a. homeomorphism from $C([\alpha'_1, \alpha'_2] \times I^{n-1})$ to $C([\alpha_1, \alpha_2] \times I^{n-1})$ that linearly transforms the former into the latter. This is defined by the formula
$$
\phi_{\alpha_1, \alpha_2}^{\alpha'_1, \alpha'_2}(y_{1..n}, t) := (( \alpha_2 - \frac{\alpha_2 - \alpha_1}{\alpha'_2 - \alpha'_1}(\alpha'_2 - y_1), y_{2..n}), t )
$$
We suppress $n$ in the notation. When $\alpha'_1=0$ and $\alpha'_2 = 1$, we simply write  $\phi_{\alpha_1, \alpha_2}$.
\end{notation}

\begin{remark}\label{rem:phiconcat}
Let $n\in \mathbb{N}$, $0 \leq \alpha_1 < \alpha_2 \leq 1$ and $0 \leq \beta_1 < \beta_2 \leq 1$. Then we have $ \phi_{\beta_1, \beta_2} \comp \phi_{\alpha_1, \alpha_2} = \phi_{\gamma_1, \gamma_2}$, where $\gamma_1 = \alpha_1 (\beta_2 - \beta_1) + \beta_1$ and $\gamma_2 = \alpha_2 (\beta_2 - \beta_1) + \beta_1 $.
\end{remark}



\begin{definition}
[Concatenation]\label{def: concatenation}
Let $(X,x_0, d)$ be a metric subanalytic germ. Let $\varphi_0$ and $\varphi_1$ be weak $b$-maps from $C(I^n)$ to $X$. 
Assume that for any point $\q$ in $C(I^{n-1})$ we have that \begin{equation}\label{eq:comp}\varphi_0\circ \nu_1\circ \q \sim_b \varphi_1 \circ \nu_0\circ \q\end{equation}
where $\nu_s:C(I^{n-1})\to C(I^n)$ are the inclussion defined by $\nu_s(y,t)=(s,y,t)$. By 
Remark~\ref{rem:glueweak}, $\varphi_0 \comp (\phi_{0, \frac{1}{2}})^{-1} $
and $\varphi_1 \comp (\phi_{\frac{1}{2}, 1})^{-1}$ glue to a weak $b$-map on $C(I^n)$, which we call the {\em concatenation of $\varphi_0$ and $\varphi_1$}. We denote it by $\varphi_0 \cdot \varphi_1$.

\end{definition}

Then, one can always concatenate $(n,b)$-loops in $MD\Gamma_n^b(X,Y,\p)$ if $n>1$, and in $MD\Gamma_n^b(X,\p)$ if $n\geq 1$. Since weak $b$-homotopies between $(n,b)$-loops can be similarly concatenated, we get that concatenation is well defined up to weak $b$-homotopies preserving the inclusions of $C(\partial I^n)$ into $Y$ and relative to $C(\partial I^n\setminus I^{n-1}\times\{1\})$.
In particular we have a group structure on the equivalence classes of $(n,b)$-loops up to homotopy whenever concatenation is possible as above. See   Proposition \ref{prop:nb-Pi} for the final statement.

\begin{notation}\label{not:const_loop}[Constant loop]
Let $(X, x_0, d, \p)$ be a pointed metric subanalytic germ and $n\in\mathbb{N}$. We denote by $c_{\p, n}$ the weak $b$-map from $C(I^n)$ to $X$ defined by $c_{\p, n}(y,t) = \p(t)$.
\end{notation}

\begin{notation}\label{not:inv}[Inverse loop]
Let $\varphi=\{(C_j,f_j)\}$ be a weak $b$-map from $C(I^n)$ to $X$. We denote the weak $b$-map $\{(\overleftarrow{C_j}, \overleftarrow{f_j})\}$ by $\overleftarrow{\varphi}$, where $\overleftarrow{C_j}$ and $\overleftarrow{f_j}$ are the result of mirroring $C_j$ and $f_j$ respectively at the $y_1$-axis:
\begin{align*}
&\overleftarrow{C_j} := \{(y_{1..n}, t) \in C(I^n): (1-y_1, y_{2..n}, t) \in C_j\},\\
& \overleftarrow{f_j}(y_{1..n}, t) := f_j(1-y_1, y_{2..n}, t)
\end{align*}
The notation $\varphi^{-a}$ stands for the result of concatenating $\overleftarrow{\varphi}$ with itself $a$ times.
\end{notation}

\begin{prop}[Definition of the (n,b)-MD homotopy groups $MD\pi_n^b(X,Y,\p)$]\label{prop:nb-Pi}
Let $(X,Y, x_0, d,\p)$ be a pointed pair of metric subanalytic germs. We denote by $MD\pi_n^b(X,Y, \p)$ the quotient of $MD\Gamma_n^b(X,Y,\p)$ by weak $b$-homotopies  that preserve the inclusion of $\partial I^{n}$ into $Y$ and are relative to $C(\partial I^n\setminus (I^{n-1}\times\{1\}))$. 
We denote by $[\varphi]$ the equivalence class in $MD\pi_n^b(X,Y, \p)$ of an element $\varphi \in MD\Gamma_n^b(X,Y,\p)$. We call it the {\em $b$-homotopy class of $\varphi$}. 

For $n>1$ concatenation of $(n,b)$-loops is possible as defined in Definition~\ref{def: concatenation}. It induces a well defined operation of weak $b$-homotopy classes as  $[\varphi_1 ]\cdot [\varphi_2]:=[\varphi_1 \cdot \varphi_2]$ that defines a group structure on $MD\pi_n^b(X,Y,\p)$. We call it the $(n,b)$-homotopy group of  $(X,Y,x_0,d,\p)$.  

For $n=1$, concatenation of $(1,b)$-loops in $MD\Gamma_1^b(X,\p)$ is always possible and it induces a group structure in $MD\pi_1^b(X,\p)$. We call $MD\pi_1^b(X,\p)$ the {\em $b$-MD fundamental group of $(X,\p)$}.

In these groups, the neutral element is $[c_{p, n}]$ and the inverse is $[\varphi]^{-1}=[\overleftarrow{\varphi}]$. 

If $n\geq 3$ the group $MD\pi_n^b(X,Y,\p)$ is abelian. If $n\geq 2$ the group $MD\pi_n^b(X,\p)$ is abelian.

\end{prop}
\begin{proof}
The proof consist in a routine checking that the corresponding homotopies can be concatenated and the usual arguments in the topological case extended to weak $b$-maps. See~\cite{thesis} for details.
\end{proof}




The proof of the equality $[\varphi]^{-1}=[\overleftarrow{\varphi}]$ yields the following slightly more general statement that will be used:  

\begin{lema}\label{lemma: existence inverse element}
Let $\varphi = \{(C_j, f_j)\}_{j \in J}$ be a weak $b$-map from $C(I)$ to $X$. Consider $\q_0 := \varphi|_{C\{0\}}$ and $\q_1 := \varphi|_{C\{1\}}$ and the associated constant loops  $C_{\q_0,1}$ and $C_{\q_1,1}$. Then there is a weak $b$-homotopy relative to $C(\partial I)$ from $\varphi \cdot \overleftarrow{\varphi}$ to $C_{\q_0,1}$ and from $\overleftarrow{\varphi} \cdot \varphi$ to $C_{\q_1,1}$.
\end{lema}



%

\subsection{Functoriality}\label{subsection:functoriality}

\begin{prop}\label{prop:functb-maps}
Let $n \in \mathbb{N}$ and $b \in (0, \infty]$. There are functorial assignments $(X,Y,x_0,d,\p) \mapsto MD\pi^{b}_1(X, Y,\p)$ 
from the category of pointed pairs of metric subanalytic germs with $b$-maps to the category $\calS$ of sets.  The assignment takes place in the categories $\calG$ or $\calA\calG$ of groups or  abelian groups when the product is defined. 
\end{prop}
\begin{proof}
This is a corollary of the fact that weak $b$-maps can be composed with $b$-maps (see Definition~\ref{def:comp_weak}).
\end{proof}

As we did in the Moderately Discontinuous Homology groups we can enrich the invariant of the MD homotopy groups giving them the structure of an object in the category $\BB-\calS$ (recall Notation \ref{not:Bcategory}). 

Let  $(X,Y,x_0,d, \p)$ be a pointed pair of metric subanalytic germs. For any $n\in \NN$ and any $b,b'\in (0, \infty]$ with $b\geq b'$ using the obvious we get a map (which respects the product whenever it is defined)
\begin{equation}\label{eq:eta_bb'}
\eta_{b, b'}:MD\pi^{b}_n(X,Y, \p) \to MD\pi^{b'}_n(X,Y, \p)
\end{equation} 

\begin{prop}
Let $n \in \mathbb{N}$. There are functorial assignments $(X,Y,x_0,d, \p) \mapsto MD\pi^{\star}_1(X,Y,\p)$ 
from the category of pointed metric subanalytic germs with Lipschitz subanalytic l.v.a. maps to $\BB-\calS$. The assignment takes place in $\BB-\calG$ when the product is defined. 
\end{prop}
\begin{proof}
This is a consequence of the previous proposition and the fact that Lipschitz l.v.a. maps are $b$-maps for any $b$. 
\end{proof}

\begin{notation}\label{notation:induced map}
Given a $b$-map  $g: (X,\p) \to (X',\p')$ we denote by $g^b_\ast$ the induced group homomorphisms $MD\pi_\ast^b(X, \p) \to MD\pi_\ast^b(X',\p')$ .
\end{notation}

\subsection{Loops which are small with respect to a dense subgerm}
\label{sec:denseqis}

This section is the analogue of Section~4.4 in~\cite{MDH}. 

\begin{definition}
\label{def:smallchain}
Let $(X,Y,d,x_0,\p)$ be a pointed pair of metric subanalytic germs and let $U$ be a subanalytic subset of $X$. 
A weak $b$-map $\varphi:C(I^n)\to X$ is  \emph{small with respect to $U$} if there exists a representative $\{(C_j,f_j\}_{j\in J}$ such that for $f_j(C_j)\subset U$  for every $k$. We denote by $MD\Gamma^{b,U}_{k}(X,Y,\p)$ the set of $(b,k)$-loops small with respect to $U$, and by
$MD\pi^{b,U}_{k}(X,Y,\p)$ the set of equivalence classes of $(b,k)$-loops small with respect to $U$, modulo $b$-homotopies relative to $C(\partial I^n\setminus (\partial I^{k-1}\times\{1\})$, preserving the inclusion of $C(\partial I^n)$ into $Y$ and which are small with respect to $U$. 
\end{definition}


\begin{prop}\label{prop:smallqis} 
Let $(X,Y,x_0,d,\p)$ be a pointed pair of metric subanalytic germs and let $U$ be a dense subanalytic subset of $X$.

For any $b<\infty$ the natural maps 
$$MD\Gamma^{b,U}_{k}(X,Y,\p)\to MD\Gamma^{b}_{k}(X,Y,\p)$$ 
$$MD\pi^{b,U}_{k}(X,Y,\p)\to MD\pi^{b}_{k}(X,Y,\p)$$ 
are bijective for any $k$.
\end{prop}

\proof
Injectivity is clear. For surjectivity, a direct application of the Claim in the proof of Proposition~46~in~\cite{MDH} allows to modify the maps $f_i$ within their $b$-equivalence class so that $f_j(C_j)\subset U$ as needed.  
\endproof

As in MD-Homology we have the following corollary:

\begin{cor}
\label{cor:closure}
Let $(X,x_0,d)$ be a metric subanalytic germ such that the metric $d$ extends to a subanalytic metric $\overline{d}$ in the closure $\overline{X}$ of $X$ in $\RR^n$. Then for any $b<\infty$ we have a bijection
$$MD\pi_k^b(X,x_0,\p)\cong MD\pi_k^b(\overline{X},x_0,\p)$$ for any $k$.
\end{cor}

\subsection{The long homotopy sequence of a pair}
\begin{theo}\label{th:longexact} Let $(X,Y, x_0, d,\p)$ be a pointed pair of closed metric subanalytic sets.
There are functorial assignments:
\begin{itemize}
\item From the category of pointed pairs of metric subanalytic sets with $b$-maps to long exact sequences of sets (respecting the product when it is defined)
$$...\to MD\pi^{b}_k(Y, \p)\to MD\pi^{b}_k(X, \p)\to MD\pi^{b}_k(X,Y,\p)\to MD\pi^{b}_{k-1}(Y, \p)\to ...$$
\item From the category of pointed pairs of metric subanalytic sets with Lipschitz subanalytic l.v.a. maps to long exact sequences of $\BB$-sets (respecting the product when it is defined)
$$...\to MD\pi^{\star}_k(Y, \p)\to MD\pi^{\star}_k(X, \p)\to MD\pi^{\star}_k(X,Y,\p)\to MD\pi^{\star}_{k-1}(Y, \p)\to ...$$
\end{itemize}
\end{theo}

The proof has the same steps as in the topological setting, but needs a preliminary result. Let us start recalling Lemma~61 of~\cite{MDH} for the convenience of the reader.

\begin{lema}
\label{lem:retracciondiscontinua}
Suppose that $S\supset Q$ are compact subanalytic subsets in $\RR^n$. Let $d$ be a subanalytic metric in $S$. There exists a partition of $S$ into finitely many disjoint subanalytic subsets $\{S_k\}_{k}$, such that there exists continuous subanalytic maps $f_k:S_k\to Q$ with the property that for any $z\in S_k$ we have the equality
\begin{equation}
\label{eq:closest}
 d(z,f_k(z))=d(z,Q).
\end{equation}
In particular $f_k(z)=z$ for any $z\in Q$. 

Moreover if $S\setminus Q$ is dense in $S$ then there exists a subanalytic stratification of $Q$ by smooth manifolds such that the union of maximal strata of the stratification by the closure relation is included in $\cup f_k(S_k\setminus Q)$. In particular $\cup f_k(S_k\setminus Q)$ is dense in $Q$.
\end{lema}

\begin{lema}
\label{lem:proyY}
Let $(X,Y, x_0, d,\p)$ be a pointed pair of closed metric subanalytic sets. Let $\varphi\in MD\Gamma^b_n(X,p)$ be a $(n,b)$-loop that has a representative $\{(D_i,g_i)\}_{i\in I}$ such that for every $i\in I$ and for every point $\p:(0,1)\to D_i$, we have that $g_i\comp\p$ is $b$-equivalent to a $b$-point in $Y$. Then there is a representative $\{(E_j,h_j)\}_{j\in J}$ of the weak $b$-map $\varphi$ such that $h_j(D_j)\subset Y$ for every $j\in J$.
\end{lema}
\proof
Apply the Lemma for $S=X$ and $Q=Y$. Let $U$ be the union of the interiors of the sets $S_k$ predicted in Lemma~\ref{lem:retracciondiscontinua}. By subdivision and Proposition~\ref{prop:smallqis} $\varphi$ has a representative  $\{(E_j,h'_j)\}_{j\in J}$ such that for every $j\in J$ we have that 
$h'_j(E_j)\subset S_{k(j)}$ for a certain $k(j)$. Notice that for every point $\p:(0,1)\to E_j$ the property that $h'_j\comp\p$ is $b$-equivalent to a point of $Y$ is still satisfied. We define $h_j:=f_{k(j)}\comp h'_j$. Then $\{(E_j,h'_j)\}_{j\in J}$ and $\{(E_j,h_j)\}_{j\in J}$ are $b$-equivalent.
\endproof

\proof[Proof of Theorem~\ref{th:longexact}]
The proof is an easy adaptation of the usual proof in the topological setting, taking the following caution: while in the topological proof, at some steps, a natural map $I^n\to Y$ would be obtained, in our setting we only obtain a weak $b$-map $C(I^n)\to X$ that has the property of transforming any point in $C(I^n)$ into a point in $X$ which is $b$-equivalent to a point in $Y$. When this happens we use Lemma~\ref{lem:proyY} to obtain a representative of the weak $b$-map mapping $C(I^n)$ into $Y$, and proceed with the usual steps of the topological proof.

We just give an example of this situation: in order to define the boundary homomorphism $MD\pi^{b}_k(X,Y,\p)\to MD\pi^{b}_{k-1}(Y, \p)$ we start with an element $\varphi\in MD\pi^{b}_k(X,Y,\p)$ and take the restriction $\varphi|_{C(I^{n-1}\times\{1\})}:C(I^{n-1})\to X$. We can not ensure that the image of 
$\varphi|_{C(I^{n-1}\times\{1\})}$ falls in $Y$ in order to obtain an element of $MD\pi^{b}_{k-1}(Y, \p)$ as needed, but using Lemma~\ref{lem:proyY} we remedy this situation.
\endproof

\subsection{$b$-connectedness, $b$-path connectedness and independence of base point.}\label{sec:conn}

%

\begin{definition}\label{def:b-pathconnected}
Let $(X, x_0, d)$ be a metric subanalytic germ. It is called {\em $b$-path connected}, if for any two points $\q_0$ and $\q_1$ in $X$, there is a weak $b$-map $\eta : C(I) \to X$ with   $\eta(0,t)=\q_0$ and $\eta(1,t)=\q_1$.
\end{definition}

Note that in order to have $\eta$ connecting $\q_0$ and $\q_1$ it is enough to have a weak $b$-map connecting some points $\q_0'$ and $\q_1'$ $b$-equivalent to the $\q_i$'s.

The concept of $b$-path connectedness is related to the concept of $b$-connectedness, which was defined in Section 9 in \cite{MDH} as follows:

\begin{definition}
\label{def:b-connected} Let $(X,x_0)$ be a metric subanalytic germ. Two connected components $X^1$ and $X^2$ of $X\setminus\{x_0\}$ are $b$-\emph{equivalent} if there exist points $\q_1$ in $X^1$ and $\q_2$ in $X^2$ that are $b$-equivalent.   
The equivalence classes are called the $b$-\emph{connected components} of $X$. The $\infty$-connected components are the usual connected components of $X\setminus\{x_0\}$.

We say $(X, x_0, d)$ is {\em $b$-connected} if it has only one $b$-connected component.
\end{definition}

They are in fact equivalent concepts: 

\begin{lema}\label{lema:connect two points}
A metric subanalytic germ $(X, x_0, d)$ is $b$-path connected if and only if it is $b$-connected.
\end{lema}
\begin{proof}
Assume $(X,x_0,d)$ is $b$-connected. Let $\p$ and $\p'$ be two points in $X$. 

Let $X^1$, $X^2$, ... , $X^N$ be the connected components of $X\setminus \{x_0\}$ (it is a finite number because $X\setminus\{0\}$ is a subanalytic set). Moreover they are subanalytic path connected (see for example Section 3.2 in \cite{Coste:1999}). Assume $\p_1:=\p$ belongs to $X^1$ and $\p'_m:=\p'$ to $X^m$. Then, by $b$-connectedness, possibly after renaming the connected components $X^i$, there exist points $\p'_1$, $\p_2$, $\p_2'$, ..., $\p_{m-1}$,  $\p_{m-1}'$,  $\p_m$ such that $\p_k$ and $\p_k'$ belong to $X^k$ and $\p_k'\sim_b \p_{k+1}$. 
Let's find weak $b$-maps $\eta_k:C(I)\to X$ connecting $\p_k$ with $\p'_{k}$   and $\eta'_k:C(I)\to X$ connecting $\p_k'$ to $\p_{k+1}$. Then, concatenating them we would have a weak $b$-map $\eta:C(I)\to X$ connecting $\p=\p_1$ and $\p'=\p'_m$.   

In general, given two $b$-equivalent points $\p$ and $\q$ in $X$, there always exists a weak $b$-map $\mu: C(I)\to X$ connecting them. Consider $\mu:=\{(C(0),\p), (C(I),f)\}$ where the map $f:C(I)\to X$ is defined as $f(y,t)=\q(t)$. So in particular we can easily find the $\eta'_k$ claimed in the paragraph before.

Consider now a point $\p:(0,\epsilon)\to X^k$. By Remark~\ref{rem:lvareduction} we can choose a conical structure, that is a subanalytic homeomorphism $h:X\to C(L_X)$, such that if its components are $h=(h_1,h_2)$, where $h_1:X\to L_X$ is the link component and $h_2:X\to (0,\epsilon)$, then $h_2(x)=||x||$. Define $\theta:(0,\epsilon)\times[0,1]\to L_X$ by the formulae $\theta(t,s):=h_1(\p(\epsilon))$ if $t\geq (1-s)\epsilon$ and $\theta(t,s):=h_1(\p(t+s\epsilon))$ if $t\leq (1-s)\epsilon$. Then $\eta:(0,\epsilon)\times[0,1]\to X$ defined by $\eta:=h^{-1}(\theta, ||\p||)$ is a weak $b$-map connecting $\p$ to a point $\q:(0,\epsilon)\to X^k$ whose image is a retraction line of the conical structure. 

Let $\p_k$ and $\p'_k$ be points in $X^k$. Connect each of them as in the previous paragraph with points $\q_k$ and $\q'_k$ whose image are retraction lines. Now using that $L_{X^k}$ is also subanalytic path-connected (see Section 3.2 in \cite{Coste:1999}), we construct easily a weak $b$-map $\eta_k$ connecting $\q_k$ with $\q'_k$. 
\end{proof}


After this, as in the topological setting we have the usual result:
\begin{prop}[Independence of base point]\label{prop:indbasepoint}
Let $(X, x_0, d)$ be a $b$-connected metric subanalytic germ. Let $\p_1$ and $\p_2$ be points in $X$ and let $\gamma$ be a weak $b$-map from $C(I)$ to $X$ connecting $\p_1$ and $\p_2$. Let $\hat{\gamma}$ be the weak $b$-map from $C(I^n)$ to $X$ defined by the formula $\hat{\gamma} := \gamma \comp \rho$, where $\rho : C(I^n) \to C(I)$ is the projection $\rho(y_{1..n}, t) := (y_1, t)$. Then the homomorphism
$$
\zeta : MD\pi_n^b(X, \p_1) \rightarrow MD\pi_n^b(X, \p_2)
$$
defined by $\zeta(\varphi) := \overleftarrow{\hat{\gamma}} \cdot \varphi \cdot \hat{\gamma}$ is an isomorphism. Moreover, its inverse is $\zeta^{-1}(\varphi) := \hat{\gamma} \cdot \varphi \cdot \overleftarrow{\hat{\gamma}}$.
\end{prop}

\subsection{Metric homotopy invariance}

There are two notions of metric homotopy under which we have invariance. We adapt the relevant definitions from~\cite{MDH}, and state the corresponding results. We do not include proofs because they are obvious consequences of the fact that weak $b$-maps can be composed with $b$-maps, or routine variations of this.

\subsubsection{Metric homotopies}

\begin{definition}[Metric homotopy]
\label{def:1homotopy}
Let $(X,Y,x_0,d,\p)$ and $(X',Y',x'_0,d',\p')$ be pointed pairs of metric subanalytic germs. Let 
$$f,g:(X,Y,x_0,d,\p)\to (X',Y',x'_0,d',\p')$$
be Lipschitz l.v.a. subanalytic maps of pointed pairs. A continuous subanalytic map $H:X \times I \to X'$ is called a {\em metric homotopy between $f$ and $g$}, if there is a uniform constant $K\geq 0$ such that for any $s$ the mapping 
$$H_s := H(-,s):(X,Y,x_0,d,\p)\to (X',Y',x'_0,d',\p')$$ is a Lipschitz l.v.a. subanalytic  of pointed pairs with  Lipschitz l.v.a. constant $K$, and $H_0=f$ and $H_1=g$. If such $H$ exists we say that $f$ and $g$ are {\em metrically homotopic}. 
\end{definition}

\begin{prop}\label{prop:homotopyinvariance}
Let $(X,Y,x_0,d,\p)$ and $(X',Y',x'_0,d',\p')$ be pointed pairs of metric subanalytic germs. 
\begin{enumerate}
\item Let $f,g:(X,Y,x_0,d,\p)\to (X',Y',x'_0,d',\p')$ be  l.v.a. subanalytic maps of pointed pairs such that there exists a continuous subanalytic mapping $H:(X,Y)\times I\to (X',Y')$ with $H_0=f$ and $H_1=g$ satisfying that there exists a uniform constant $K>0$ such that for every $s$, the mapping $H_s$ is l.v.a. for the constant $K$. Then we have that both $f_k^\infty,g_k^\infty:MD\pi^\infty_k(X,Y,\p)\to MD\pi^\infty_k(X',Y',\p')$ are the same map for any $k$.
 \item\label{metric homotopy invariance} Moreover, if $f$ and $g$  are metrically homotopic. Then 
 $$f^{\star}_k,g^{\star}_k:MD\pi^\star_k(X,Y,\p)\to MD\pi^\star_k(X',Y',\p')$$ are equal for any $k$ and $\star$.
\end{enumerate}
\end{prop}

\subsubsection{$b$-homotopies}
Let $(X,Y,x_0,d,\p)$ be a pointed pair of metric subanalytic germs. Assume that the metric $d$ satisfies $d_{out}\leq d$ (inner and outer metrics satisfy this). Recall from~\cite{MDH} (Definition~82), the definition of the metric subanalytic germ $X\times_\p I$. Obviously $Y\times_\p I$ is a subgerm of $X\times_\p I$, and we consider in it the metric induced from $X\times_\p I$.

\begin{definition}[$b$-homotopy]\label{def:b-homotopy}
Let $(X,Y,x_0,d,\p)$ and $(X',Y',x'_0,d',\p')$ be pointed pairs of metric subanalytic germs.metric subanalytic germs. A {\em b-homotopy} between them is a $b$-map from $X\times_\p I$ to $X'$ such that there exists a representative $\{(C_i,f_i)\}_{i\in I}$ satisfying
\begin{enumerate}
 \item for any point $\p$ with image inside $C_i \cap Y\times_\p I$, the point $f_i\comp \p$ is $b$-equivalent to a point in $Y'$,
 \item if $\frak{r}:(0,\epsilon)\to X\times_\p I$ projects to the point $\p$ under the natural projection to $X$, then then $f_i\comp \frak{r}$ is $b$-equivalent to $\p'$.
\end{enumerate}
\end{definition}

\begin{prop}\label{prop:homotopyinvariance2} If there is a $b$-homotopy $H$ with $H_0=f$ and $H_1=g$,  then  
 $$f_k^b,g_k^b:MD\pi_k^{b}(X,Y,\p)\to MDC^b_k(X',Y',\p')$$
 represent the same map for any $k$. 
\end{prop}

\subsection{The metric Hurewicz homomorphism}\label{subsection:Hurewiczmorphism}
In the same way as in the topological homotopy and homology theories, for the $b$-MD homology and homotopy theories there is a Hurewicz homomorphism relating them.

We recall that a simplex in $(X,0)$  in the  Moderately Discontinuous Homology is a  l.v.a. subanalytic mapping germ $\sigma:\hat{\Delta}_n\to (X,0)$ being $\hat{\Delta}_n$ the cone over the standard simplex $\Delta_n$. We say that two $n$-simplices $\sigma$, $\sigma':\hat{\Delta}_n\to (X,0)$ are $b$-equivalent in the metric subanalytic germ $(X,0,d)$ if for every point $p$ in $\hat{\Delta}_n$ we have that $\sigma\circ p$ is $b$-equivalent to $\sigma'\circ p$. 

The $b$-MD Chain group $MDC^b_n(X)$ in dimension $n$ is the quotient of the group of formal finite sums of classes of simplices up to $b$-equivalence by the homological subdivision equivalence relation (see Definition 22 in Section 3 in \cite{MDH}), which essentially comes from identifying a simplex $\sigma:\hat{\Delta}_n\to (X,0)$ with any formal sum $\sum_i\sigma\circ\rho_i$ where the $\rho_i:\hat\Delta_n\to \hat\Delta_n$ are orientation preserving l.v.a. subanalytic parametrizations of the simplices of a finite subanalytic triangulation of $\hat{\Delta}_n$.  
  
With this in mind, we define a map from $(n,b)$-loops, or more generally weak $b$-maps, to the $b$-MD chains of a pointed metric subanalytic germ $((X, x_0, d), p)$:
\begin{equation}\label{eq:domainhurewiczmor}
\zeta_{n, b} : \{\varphi : \varphi \text{ is a weak } b \text{-map from } C(I^n) \text{ to } X \} \to MDC_n^b(X; \mathbb{Z})
\end{equation}
as follows:
 given $\varphi = \{(C_j, f_j)\}_{j \in J}$ be a weak $b$-map from $C(I^n)$ to $X$ we define the map $\zeta$ by the formula
\begin{equation}\label{eq:defhurewiczmor}
\zeta_{n, b} ( \varphi ) := \sum_{k \in K} f_{r(k)} \comp \rho_k
\end{equation}
where $\{\rho_k:\hat{\Delta}_n\to (C(I^n),0)\}_{k \in K}$ is an orientation preserving homological subdivision (recall Definition~18~of~\cite{MDH}) of $C(I^n)$ whose associated triangulation is compatible with $\{C_j\}_{j \in J}$ and  $r(k) \in J$ is such that the image of $\rho_k$ is contained in $C_{r(k)}$ for every $k$.
\begin{lema}\label{lema:hurewiczmap}
Let $(X,x_0,d,\p)$ be a pointed metric subanalytic germ. Let $\zeta_{n, b}$ be as defined above. Then $\zeta_{n, b}$ has the following properties:
\begin{enumerate}
\item The map $\zeta_{n,b}$ is well-defined independent of the choice of the homological subdivision.

\item If $\varphi_1$ and $\varphi_2$ can be concatenated then we have $\zeta_{n,b}( \varphi_1 \cdot \varphi_2)= \zeta_{n,b}(\varphi_1) + \zeta_{n,b}(\varphi_2)$.
\item\label{item:hurewiczmorphisminverse} We have $\zeta_{n,b}(\overleftarrow{\varphi}) = - \zeta_{n,b}(\varphi)$.
\end{enumerate}
\end{lema}
\begin{proof}
All the stated properties follow from the homological subdivision equivalence in $MDC_n^b((X, \text{Im}(\p)); \mathbb{Z})$. In particular, for property (\ref{item:hurewiczmorphisminverse}) recall Remark~20 of~\cite{MDH}.
\end{proof}

\begin{prop}[Metric Hurewicz homomorphism]\label{prop:hurewiczmorphism}
Let $((X,x_0,d,\p)$ be a pointed metric subanalytic germ. Let $b \in (0, \infty]$ and $n \in \mathbb{N}$. Then the restriction of $\zeta_{n, b}$ to the space of $(n,b)$-loops induces a homomorphism 
$$
\overline{\zeta_{n, b}}: MD\pi_n^b(X,\p) \to MDH_n^b(X;\mathbb{Z}),
$$
which we call the Hurewicz morphism.
\end{prop}
\begin{proof}
By Lemma~\ref{lema:hurewiczmap}, if $\overline{\zeta_{n, b}}$ is well-defined, then it is a homomorphism. 

We consider the relative homology complex $MDC_n^b(X, \text{Im}(\p); \mathbb{Z})$ (See section~3.5 of~\cite{MDH}). Since $MDH_n^b(X;\mathbb{Z})\to 
MDH_n^b(X, \text{Im}(\p); \mathbb{Z})$ is an isomorphism for any $n>0$ we can consider the composition of $\zeta_{n, b}$ with the natural projection to $MDC_n^b(X, \text{Im}(\p); \mathbb{Z})$ instead of $\zeta_{n, b}$ itself. Then the image of any $n$-loop in $MD\pi_n^b(X,\p)$ is obviously a cycle, and homotopic $n$-loops give rise to homologous cycles. Moreover given the constant $b$-loop $c_{\p,n}$ we have $\zeta_{n,b}(c_{\p,n})=0$.
\end{proof}

\begin{theo}[Metric Hurewicz isomorphism]
\label{theo:hur}
Let $((X,x_0,d,\p)$ be a pointed metric subanalytic germ. Let $b \in (0, \infty]$. 
If $X$ is $b$-connected then the group $MDH_1^b(X;\ZZ)$ is the abelianization of $MD\pi_1^b(X,\p)$, and the Hurewicz homomorphism is the abelianization map. 

Assume that $MD\pi_1^b(X,\p)$ is trivial and that $d$ is either the inner or the outer metric. If $MD\pi_k^b(X,\p)$ is trivial for $k<n$ then $\overline{\zeta_{n, b}}: MD\pi_n^b(X,\p) \to MDH_n^b(X;\mathbb{Z})$ is a group isomorphism.
\end{theo}
\begin{proof}
The proof of the fundamental group part is an adaptation of the usual topological proof (see~\cite{Hatcher}) to our setting, and it involves no new ideas. The reader interested in the details may consult~\cite{thesis}. The higher homotopy groups case is harder and is provided at the end of Section~\ref{sec:hurew}. The proof in Section~\ref{sec:hurew} also contains the fundamental group statement for the inner and the outer metrics.

\end{proof}

\section{The metric Seifert van Kampen Theorem for the MD-Fundamental groupoid.}\label{sec:svk}

\subsection{The MD Fundamental Groupoid.}\label{sec:MDgroupoid}

Along this section we use the following terminology to study a metric subanalytic germ $(X,0,d)$. 
Given two $b$-points $\p$ and $\q$ in $X$, a $b$-path from $\p$ to $\q$ is a weak $b$-maps $\gamma: C(I)\to (X,0)$ such that $\gamma(0,\_)\sim_b \p$ and $\gamma(1,\_)\sim_b\q$. Two $b$-paths $\gamma$ and $\beta$ from $\p$ to $\q$ are $b$-homotopy equivalent if there exists a homotopy $\eta:C(I\times I)\to X$ relative to the extremes  $C(\{0,1\}\times I)$ such that $\eta|_{C(I\times \{0\})}$ is $b$-equivalent to $\gamma$ and $\eta|_{C(I\times \{1\})\}}$ is $b$-equivalent to $\beta$. 

\begin{definition}\label{def:MDgroupoid} Let $(X,0,d)$ be a metric subanalytic germ. The $b$-MD Fundamental Groupoid of $X$  is  the category, denoted by $MD\pi_1^b(X,d)$ or $MD\pi^b_1(X)$, whose objects are the $b$-points of $(X,0)$ and whose morphisms from $\p$ to $\q$ are the $b$-homotopy classes of $b$-paths from $\p$ to $\q$ relative to the extremes.
    
Let $\fP$ be a subset of  $b$-points in $X$. We denote by $MD\pi_1^b(X,d,\fP)$ or $MD\pi_1^b(X,\fP)$ the full subcategory of $MD\pi_1^b(X)$ whose set of objects is $\fP$. 
\end{definition}

Composition in $MD\pi_1^b(X,\fP)$ is given by concatenation (see Definition \ref{def: concatenation}) and every morphism is invertible: the inverse of the $b$-path class of $\gamma$ has $\overleftarrow{\gamma}$ as a representative (see Notation \ref{not:inv}). When $\fP$ is a single point $\{\p\}$ we recover the $b$-MD  fundamental group (see Proposition  \ref{prop:nb-Pi}).

Analogously to Proposition \ref{prop:functb-maps} we have the following:
\begin{prop} For every $b\in (0,+\infty)$, the $b$-MD Fundamental Grupoid $MD\pi^b_1$ is a functor from the category of metric subanalytic germs (with both Lipschitz l.v.a. maps as morphism or $b$-maps) to the category of groupoids. Moreover, for $b\geq b'$ there are groupoid morphisms, that we denote as in (\ref{eq:eta_bb'})
\begin{equation}\label{eq:eta_bb'_GP}
\eta_{b, b'}:MD\pi^{b}_1(X, \fP) \to MD\pi^{b'}_1(X, \fP).\end{equation} 
So, we have the  $b$-MD Fundamental Groupoid $MD\pi^b_1$ as a functor from the category of metric subanalytic germs (with both Lipschitz l.v.a. maps as morphism or $b$-maps) to the category $\BB-\calG\calP$. 
\end{prop} 

We state a basic observation that will be used later: 

\begin{lema}\label{lem:tec_conc} Let $(X,0,d)$ be a metric subanalytic germ. Let $\fP$ be a set of $b$-points of $X$. Let $\gamma:C(I)\to X$ be a $b$-path with extremes in points of $\fP$. Then, given any subanalytic homeomorphism $h:C(I) \to C(I)$ which resticts to the identity in $C(\{0\})$ and $C(\{1\}$ we have that $[\gamma\circ h]=[\gamma]$ in $MD\pi_1^b(X,\fP)$.
\end{lema}
\begin{proof} 
Consider the homotopy $H:C(I)\times I\to C(I)$ defined by $H(yt,t,s):=(1-s)(yt,t)+s h(yt,t)$ and compose it with $\gamma$.  
\end{proof}

\subsection{The metric Seifert- van Kampen Theorem.}\label{subsec:svk***}
In this section we consider a covering $\{U_i\}_i$ of a metric subanalytic germ $(X,0,d)$. We will always considered the subsets $U_i$ with metrics $d_i$ where $d_i$ are defined by one of the following cases:
\begin{itemize}
\item[(a)] either the $d_i$'s are the restriction metrics $d|_{U_i}$
\item[(b)] or the $d_i$'s are the inner metrics in $U_i$ induced by the infimum of lengths of rectifiable paths in $(X,d)$. A particular case is when $d$ is the outer or inner metric in $X$ induced by the euclidean metric in $\RR^n$ and then the $d_i$'s are the inner metrics induced in every $U_i$.
\end{itemize} 

\begin{definition}
\label{def:condition*}
Let $\calU=\{U_i\}_{i\in I}$ be a finite cover of a metric subanalytic germ $(X,0,d)$ by subanalytic open subsets $U_i$ endowed with metrics $d_i$ (all either of type (a) o (b) above). Let $\fP$ be a set of $b$-points in $X$. The pair $(\calU,\fP)$ satisfies condition  

\begin{itemize}
\item [$(*)_b$] if for every $b$-equivalent points $\q_i, \q_j$ and $\q_k$ in $U_i$, $U_j$ and $U_k$ respectively  (where $i$, $j$ and $k$ may coincide), there exist 

\begin{itemize}
\item a point $\p$ in $\fP$, 
\item  a $b$-path $\delta_v:C(I)\to (U_v,d_v)$   joining $\q_v$ with  $\p$ for every $v=i,j,k$   , 
\item  weak $b$-homotopies $\mu_{vw}:C(I^2)\to (U_v,d_v)$ for every  $v,\ w\in\{i,j,k\}$ relative to the extremes $\p$ and $\q_v$, with $\mu_{vw}|_{C(I\times \{0\})}\sim_b \delta_v$ in $(U_v,d_v)$ and such that for every $v,w=i,j,k$ the $b$-paths $\mu_{vw}|_{C(I\times\{1\})}$ and $\mu_{wv}|_{C(I\times\{1\})}$ are $b$-equivalent as $b$-maps in $(X,d)$. 
\end{itemize}
\end{itemize}
\end{definition}


\begin{notation}
Given a metric subanalytic germ $(X,0,d)$,  a subanalytic subgerm $U\subset X$ and  a set $\fP$ of $b$-points in $X$ we denote by $\fP\cap U$ the subset of points in $\fP$ which are $b$-equivalent in $(X,d)$ to a point in $U$. Then we can consider $\pi_1^b(U,\fP\cap U)$ which we will denote simply by $\pi_1^b(U,\fP)$.  Given $V\subset U\subset X$ considered with metrics $d_V$ and $d_U$ of the same type with respect to Definition \ref{def:condition*}, functoriality gives rise to a morphism of groupoids
$$\iota_{V,U}: MD\pi_1^b(V,\fP)\to \pi_1^b(U,\fP).$$
\end{notation}

\begin{definition}
\label{def:propdagger}
Let $(X,0,d)$ be a metric subanalytic germ, $\calU=\{U_i\}_{i\in I}$ a finite cover and $\fP$ a set of $b$-points in $X$.
A groupoid $K$ and a collection of groupoid morphisms $a_i:MD\pi_1^b(U_i,\fP)\to K$ with $i\in I$ satisfy property 
\begin{itemize}
\item[$(\dagger)_b$]
if for any $[\gamma_i]\in Mor_{MD\pi^b_1(U_i,\fP)}(\p,\q)$ and $[\gamma_j]\in Mor_{MD\pi^b_1(U_j,\fP)}(\p,\q)$ 
having representatives $\gamma_i$ en $U_i$ and $\gamma_j$ en $U_j$ which are equivalent as $b$-maps in $X$ we have the equality $a_i([\gamma_i])=a_j([\gamma_j)]$.
\end{itemize}
\end{definition}

\begin{theo}[A universal property characterizing the MD-fundamental groupoid]
\label{theo:1stuniversal}
Let $(X,0,d)$ be a metric subanalytic germ and $(\calU,\fP)$ be a finite cover and a set of $b$-points in $X$ satisfying condition $(*)$. 
There is a unique (up to groupoid isomorphism) groupoid $L$ and  groupoid morphisms $ b_i:MD\pi_1(U_i,\fP)\to L$ satisfying property $(\dagger)_b$ such that for any other groupoid $K$ and groupoid morphisms $\kappa_i:MD\pi_1^b(U_i,\fP)\to K$ with the property $(\dagger)_b$ there exists a unique groupoid morphism $\kappa:L\to K$ such that $\kappa\comp b_i=\kappa_i$ for any $i\in I$. Moreover 
$MD\pi_1^b(X,\fP)$ and the  morphisms $\iota_{U_i,X}$ coincide with $L$ and $b_i$ for any $i\in I$.

In other words and informally: the fundamental groupoid $MD\pi_1^b(X,\fP)$ is the unique initial object of the category of groupoids having property $(\dagger)_b$.
\end{theo}

\begin{proof}
Given the existence, the uniqueness is obvious. It is also clear that $MD\pi_1^b(X,\fP)$ and $\iota_{U_i,X}$ satisfy property $(\dagger)$. Consider a groupoid $K$ and groupoid morphisms $\kappa_i:MD\pi_1^b(U_i,\fP)\to K$ with property $(\dagger)$. We have to find the predicted groupoid morphism $\kappa:MD\pi_1^b(X,\fP)\to K$ and to prove that $\kappa$ is unique. 

The definition and uniqueness of $\kappa$ at the level of objects follows from condition $(\dagger)$ applied to the constant loops (that are the unit morphisms in the fundamental groupoid), the bijection between objects and unit elements of the isotropy groups of a groupoid and the fact that groupoid morphisms preserve unit elements.

In the rest of the proof we prove that there is a unique possible definition of $\kappa$ at the level of morphisms. We adapt the line of the proof of Seifert-van Kampen for groupoids in \cite{BR} (see also \cite{Hatcher} for the fundamental group case) in the topological setting, with some additional non-trivial arguments. 

Given  $\p,\q\in \fP$ we define the map $\kappa:Mor_{MD\pi_1^b(X,\fP)}(\p,\q)\to Mor_{K}(\kappa(\p),\kappa(\q))$. Given an element $[\gamma]\in Mor_{MD\pi_1^b(X,\fP)}(\p,\q)$, if  
\begin{equation}\label{eq:concat}[\gamma]=[\gamma_{1}]\centerdot...\centerdot [\gamma_{r}]
\end{equation} where
$[\gamma_l]$ are 
elements of $Mor_{MD\pi_1^b(U_{i_l},\fP)}(\p_{l},\p_{l+1})$ for certain $i_l\in I$ and $\p_l$, $\p_{l+1}\in \fP$, then necessarily  
\begin{equation}\label{eq:kappa}\kappa([\gamma]):=\kappa_{i_1}([\gamma_{1}])\centerdot...\centerdot \kappa_{i_r}([\gamma_{r}]).
\end{equation}
Notice that if any $[\gamma_l]$ 
has a $b$-map representative $\gamma_l'$ with image contained in a different $U_{i_l'}$, then  property $(\dagger)$ of the system of morphisms $\{\kappa_i\}_{i\in I}$ shows
the equality $\kappa_{i_l}([\gamma_{l}])=\kappa_{i_l'}([\gamma_{l}'])$ so the value obtained for $\kappa([\gamma])$   is independent of the choice of the subset $U_{i_l}$. 

Notice that any $b$-loop $[\gamma]$ admits an expression as (\ref{eq:concat}). Indeed, let $\{(C_j,f_j)\}_{j\in J}$ be a representative of $\gamma$ such that such that for each $j$ there is a $i_j$ such that $f_j(C_j)\subset U_{i_j}$ (this can be achieved by taking the refinement $\{C_j\cap f_j^{-1}(U_i)\}_{i,j}$ of any covering $\{C_j\}_j$). By the existence of subanalytic triangulations and Remark~\ref{rem:lvareduction} there is a triangulation $h:|K|\to C(I)$ compatible with the decomposition by $\{C_j\}_j\in J$ and preserving $t$-levels $C_i$. Applying Lemma~\ref{lem:tec_conc} we have that $\gamma\circ h$ is a representative of the homotopy class $[\gamma]$, and it is clear that it is a concatenation of $b$-paths $\gamma_0\cdot...\gamma_r$, each of which has image contained in a subset $U_i$ of the cover. Now we modify each $\gamma_{j}$ so that its starting and ending points belong to $\fP$. Since the ending point $\frak{e}_j$ of $\gamma_{j}$ and the starting point $\frak{s}_{j+1}$ of $\gamma_{j+1}$ coincide as $b$-points in $X$, by condition $(*)$ there exists a point $\p_{j+1}\in \fP$, $b$-paths $\delta_{j+1}^0$ and $\delta_{j+1}^1$ connecting $\frak{e}_j$ with $\p_{j+1}$ inside $U_{i_j}$ and $\frak{s}_{j+1}$ with $\p_{j+1}$ inside $U_{i_{j+1}}$ respectively, and so that  $\delta_{j+1}^0\cdot\overleftarrow{\delta_{j+1}^1}$ is weak $b$-homotopic to a constant $b$-path. Then we have 
$$[\gamma]=[\gamma_0\cdot\delta_1^0]\cdot [\overleftarrow{\delta_1^1}\cdot\gamma_1\cdot\delta_2^0]\cdot...\cdot[\overleftarrow{\delta_r^1}\cdot\gamma_r],$$
which is the needed expression of the form (\ref{eq:concat}).

In order to finish the proof, we have to see that for any two expressions of an element $[\gamma]$ as in $(\ref{eq:concat})$, the values $\kappa([\gamma])$ induced as in (\ref{eq:kappa}) are the same: as a conclusion, we get that $\kappa$ is well defined.

Consider two expressions $[\gamma]=[\gamma_1\cdot...\cdot \gamma_r]=[\gamma'_1\cdot...\cdot \gamma'_{r'}]$ as in $(\ref{eq:concat})$.
Let $\eta$ be a $b$-homotopy connecting $\gamma_1\cdot...\cdot \gamma_r$ and $\gamma'_1\cdot...\cdot \gamma'_{r'}$ relative to the extremes, given by a representative $\{(D_k, g_k)\}_{k\in A}$ with the $D_k\cap D_l$ with empty interior and such that $g_l(D_l)\subset U_{i(l)}$ for certain $i(l)\in I$. Since we have the $b$-equivalences 
$\eta|_{C(I\times\{0\})}\sim_b \gamma_1\cdot...\cdot \gamma_r$ and 
$\eta|_{C(I\times\{1\})}\sim_b \gamma'_1\cdot...\cdot \gamma'_{r'},$
after a refinement of the partition, we may assume that $\{(D_k, g_k)\}_{k\in A}$ refines the decompositions induced in $C(I\times\{0\})$ and $C(I\times\{1\})$ by the concatenation expressions. Next we will  modify the homotopy $\eta$ and the decomposition $\{(D_k, g_k)\}_{k\in A}$ in several steps so that
\begin{enumerate}
 \item[(i)] the $D_k$ are cones over convex polygons in $I_2$ that only intersect along faces,
 \item[(ii)] no more than three $D_k$ meet in a point.
\end{enumerate}
The condition (ii) is crucial because we will need to apply property (*), which is formulated for coverings allowing only $3$-fold intersections. The first condition is very convenient to easily express certain restrictions of the homotopy (over the faces of the $D_k$'s) as a concatenation of paths.

We consider a triangulation $\tau:C(|K|)\to C(I^2)$ where $K$ is a simplicial complex (and $|K|=I^2$), adapted to the natural strata of $C(I^2)$ and to the $\{D_{k}\}_k$ and that preserves $t$-levels. Then, the homotopy $\widetilde{\eta}:=\eta\circ\tau$ is a weak $b$-homotopy between the concatenations of reparametrizations $\widetilde{\gamma}_i$ of the $b$-paths $\gamma_i$ and $\widetilde{\gamma}'_i$ of $\gamma'_i$. Let $U_{j(i)}$ be a subset of the cover containing the image of $\gamma_i$. By Lemma~\ref{lem:tec_conc} the path $[\gamma_i]$ represents the same element in $MD\pi_1^b(U_{j(i)},\fP)$ than $\widetilde{\gamma}_i$. So, by $(\dagger)$ we have 
\begin{equation}\label{eq:tilde}\kappa[\gamma_1\cdot \gamma_2\cdot...\cdot \gamma_r]=\kappa[\widetilde{\gamma}_1\cdot...\cdot \widetilde{\gamma}_r]\ \ \text{and}\ \ \kappa[\gamma'_1\cdot ...\cdot \gamma'_{r'}]=\kappa[\widetilde{\gamma}'_1\cdot ...\cdot \widetilde{\gamma}'_{r'}].
\end{equation}

So, we have to check the equality of the two second terms of each of the previous equalities, and we can work with the homotopy $\widetilde{\eta}$. We denote by $\{T_l\}_{l}$ the family of 2-dimensional simplices of $K$, which decompose $|K|=I^2$.

In case there are vertices in $K$ that meet more than three 2-dimensional simplices, we do the following modifications of the partition $\{T_l\}_l$ of $I^2$ and of $\widetilde{\eta}$.
For every vertex $v$ of valency $m>3$ we consider a convex $m$-gonal piece $G_v$ (the union of $m$ small triangular pieces each of them inside one of the adjacent 2-dimensional simplices) as in Figure \ref{fig:polygon}. We choose them small enough so that they are disjoint as in the figure. We consider the partition $\calP$ of $I^2$ given by  $\{G_v\}_{val(v)>3}$ and $\{\widetilde{T}_i\}_i$ with $\widetilde{T}_i:=T_i\setminus \cup_v G_v$. We consider the 
conical partition $C(\calP)$ induced on $C(I^2)$. We now modify the homotopy $\widetilde{\eta}$. Consider the continuous mapping $\tau_2:\calP\to |K|$ as follows: 
\begin{itemize}
\item the restriction to every $G_v$ collapses, 
to $v$
\item it is the identity in the rest of the vertices,
\item the restriction to every $\widetilde{T_i}$ has image $T_i$ and is any subanalytic map satisfying the following conditions: it is a homeomorphism of the interiors and at the boundary it interpolates affinely the definition provided in previous 2 items. 
\end{itemize}

\begin{figure}
\includegraphics[scale=0.2]{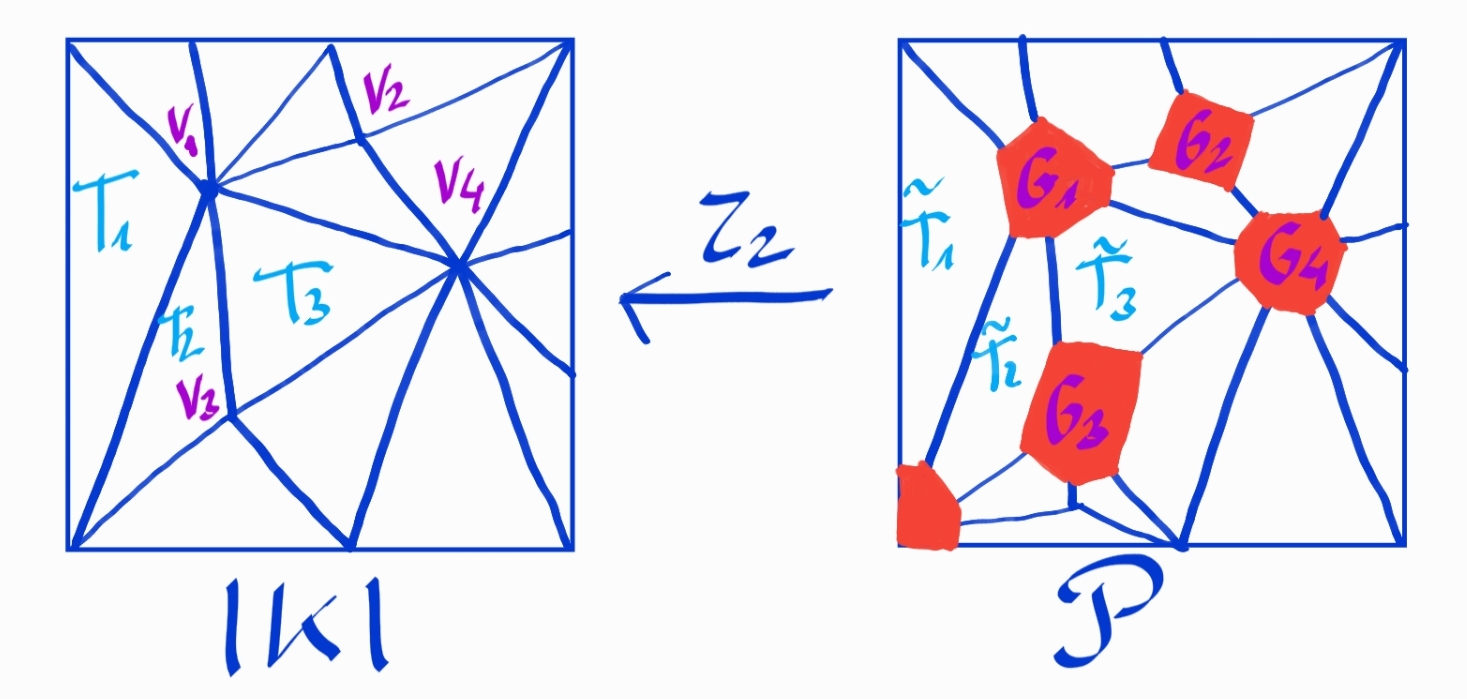}
\caption{The refinement $\calP$ of the partition $|K|$ of $I^2$. }\label{fig:polygon}
\end{figure}

We also denote by $\tau_2$ the induced conical map from $C(\calP)$ to $C(|K|)$ given by $(y,t)\mapsto (\tau_2(y),t)$. We define a $b$-homotopy $\xi:=\widetilde{\eta}\circ\tau_2$ as follows. For every region $C(\widetilde{T}_i)$ there is a $D_{k}$ such that $\tau(C(T_i))\subset D_k$. Then we define $\xi|_{C(\widetilde{T}_i\})}:=g_k\comp\tau\comp\tau_2$. At a piece $C(G_v)$ there are several $D_{k}$'s such that $\tau_2(C(G_v))\subset D_k$. Choose one of them and define $\xi:=g_k\comp\tau\comp\tau_2$ there. The reader may check easily that we obtain a $b$-map by this procedure. By abuse of notation we denote also by $\{D_k\}_k$ its associated covering $\{C(\widetilde{T}_i)\}_i\cup \{C(G_v)\}$.

The replacement of the homotopy $\widetilde{\eta}$ by $\xi$ performed above induces a replacement of the concatenations
$[\widetilde{\gamma}_1\cdot...\cdot \widetilde{\gamma}_r]$
and $[\widetilde{\gamma}'_1\cdot ...\cdot \widetilde{\gamma}'_{r'}]$ to which $\widetilde{\eta}$ resticts to $C(I\times\{0\})$ and $C(I\times\{1\})$. The restriction of $\xi$ to $C(I\times\{0\})$ is the concatenation $[\widetilde{\gamma}_1\cdot...\cdot \widetilde{\gamma}_r]$ replaced by the same concatenation in which several constant $b$-loops (corresponding to the $G_v$ regions meeting $C(I\times\{0\})$) are inserted. Similarly for $[\widetilde{\gamma}'_1\cdot ...\cdot \widetilde{\gamma}'_{r'}]$. Since the value of $\kappa$ is unchanged under this operation we can safely assume that the homotopy connecting the concatenations is $\xi$, which satisfies (i) and (ii).

To finish the proof we need to check the equality 
\begin{equation}\label{eq:xi}\kappa[\widetilde{\gamma}_1\cdot...\cdot \widetilde{\gamma}_r]=\kappa[\widetilde{\gamma}'_1\cdot ...\cdot \widetilde{\gamma}'_{r'}].
\end{equation}

Now, we modify $\xi$,  analogously to the classical proof \cite{BR}, so that it sends all the meeting points of 3 cells in the decomposition $C(\calP)$ to points in $\fP$. Note that in the case of points over $C(I\times\{0\})$ or $C(I\times\{1\})$ they are already points in $\fP$, so we do not need any modification there. Let $\q$ be the point in $C(I^2)$ whose image is the intersection of $3$ cells $D_{u}, D_{v}, D_{w}$ of the decomposition $C(\calP)$. Recall that the points $g_{u}\comp \q$, $g_{v}\comp \q$ and $g_{w}\comp\q$ are pairwise $b$-equivalent. By condition (*) there exist 
\begin{itemize}
 \item a $b$-point $\p\in \fP$,
 \item $b$-paths $\delta_u:C(I)\to U_{j(u)}$  joining $g_{u}\comp \q$ with $\p$, $\delta_v:C(I)\to U_{j(v)}$  joining $g_{v}\comp \q$ with $\p$ and $\delta_w:C(I)\to U_{j(w)}$  joining $g_{w}\comp \q$ with $\p$,
 \item weak $b$-homotopies $\mu_{vw}:C(I^2)\to U_{j(v)}$ for every pair $v,w\in\{u,v,w\}^2$ of different indexes with $\mu_{vw}|_{C(I\times \{0\})}\sim_b \delta_v$ and such that the condition in (*) holds.
\end{itemize}
We modify the $b$-homotopy $\xi$ in a neighbourhood of every intersection of every (two or) three cells $D_k$ according to Figure \ref{fig:verticesinP}, gluing the $b$-homotopies $\mu_{ij}$ and some composition of $\xi$ with some collapsing mapping similar to the $\tau_2$ used above. This is just as the standard procedure in classical topology, plus the observation that condition (*) is what we need so that we actually obtain a weak $b$-map. We leave the details to the reader. We call $\xi'$ the resulting $b$-homotopy. 

\begin{figure}
\includegraphics[scale=0.2]{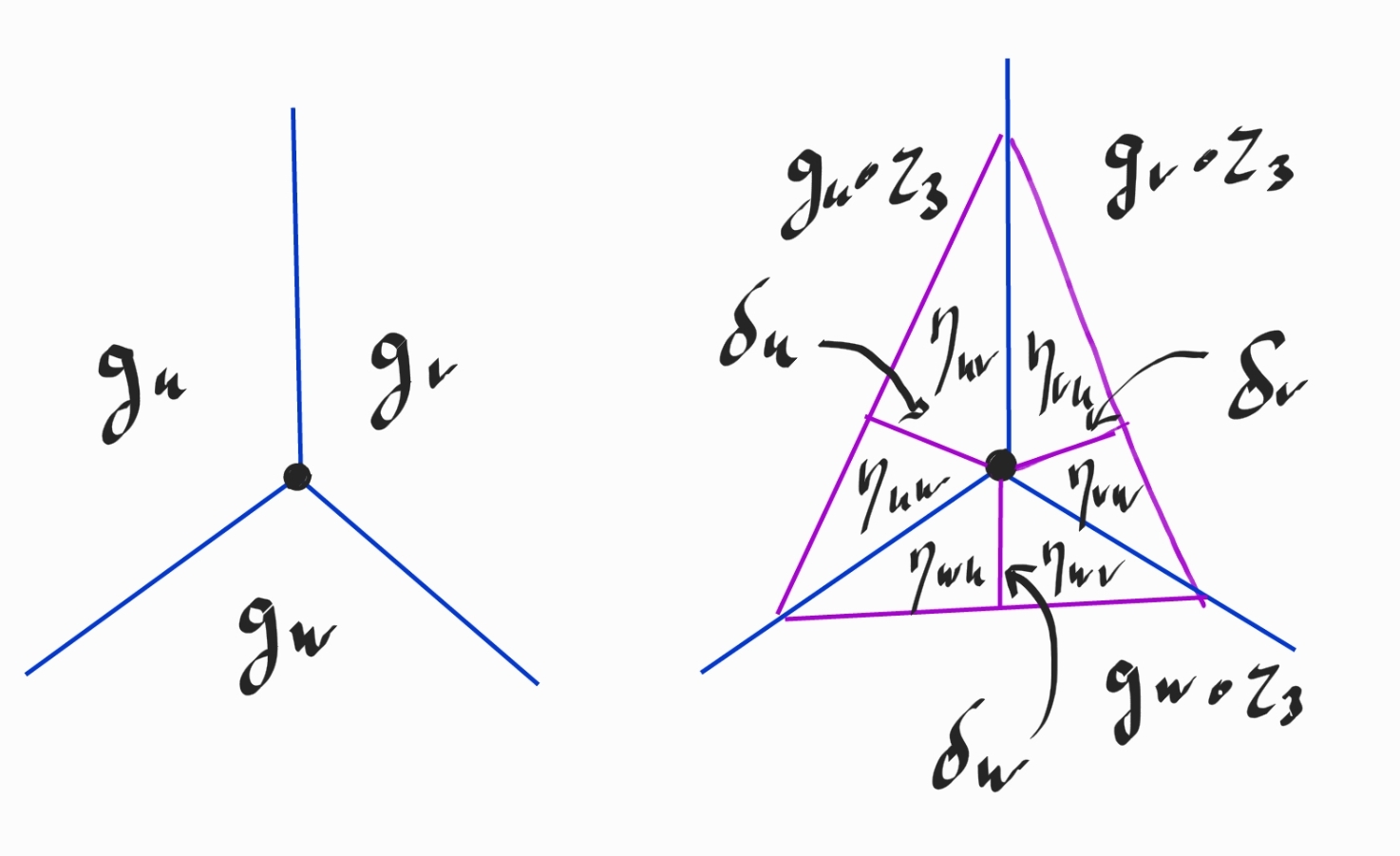} \caption{Modification  around the vertices of the decomposition $C(\calP)$, using hypothesis (*), to obtain the weak $b$-homotopy $\xi'$.}\label{fig:verticesinP}

\end{figure}

\begin{figure}
\includegraphics[scale=0.13]{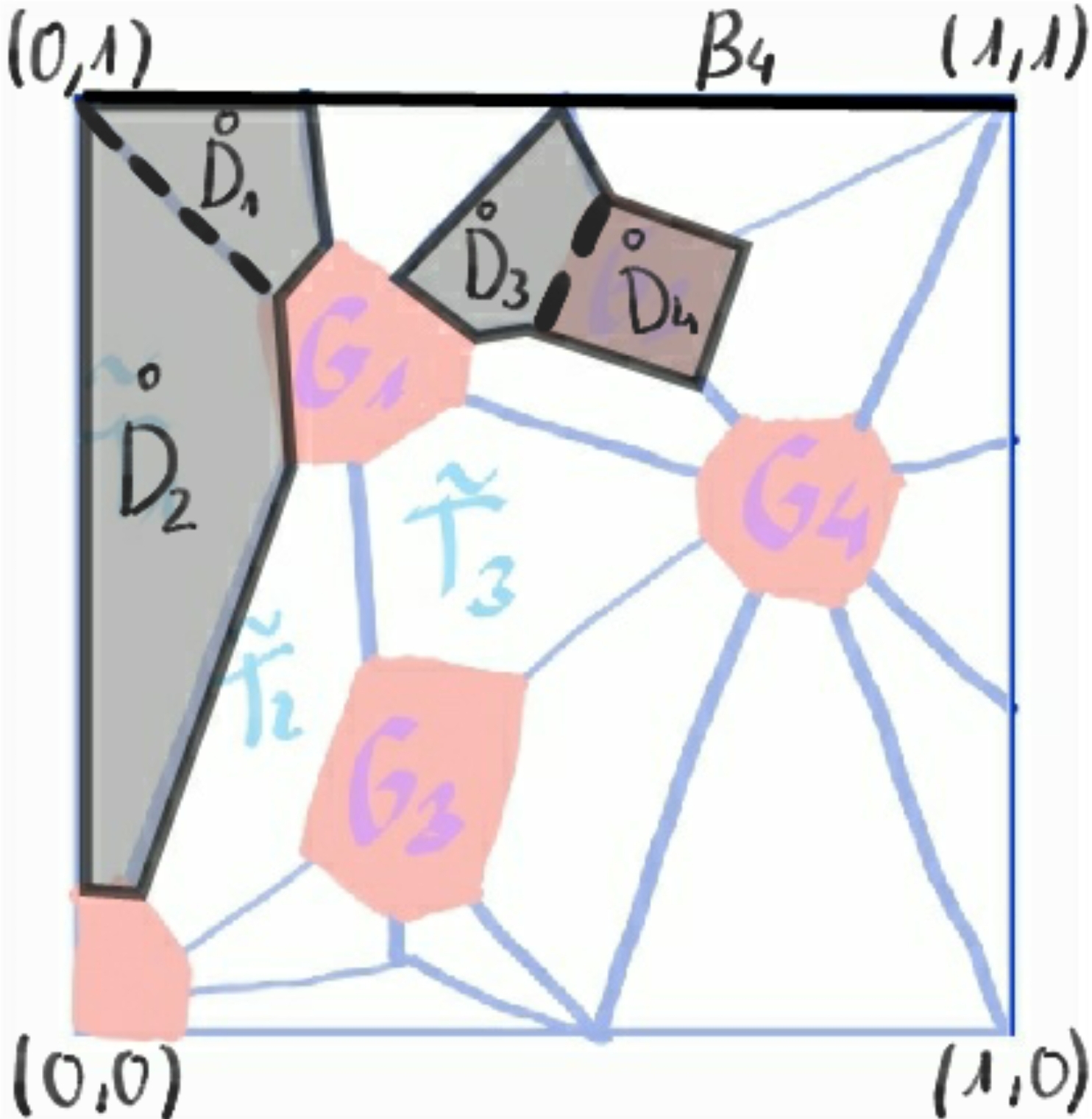}
\caption{Possible choice for the polygonal line $B_4$ for the example in the previous figures.}\label{fig:Bn}
\end{figure}

Now we adapt the procedure of the classical proof (see \cite{Hatcher}). 
 We choose a level $t=t_0$ in $C(I^2)$ and denote by $D_j^0$ the intersection of $D_j$ with that level. Considering a numbering $D^0_1,....,D^0_m$ of the polygons which cover $I^2$ such that for any $n\leq m$ the union $Q_n:=(I\times \{1\})\cup \cup_{i=1}^nD^0_n$ is contractible.
For any $n$ we define $B_n$ the {\em lower boundary} of $Q_n$, that is the connected  polygonal line containing the points $(0,1)$ and $(1,1)$ defined as the closure of the boundary of $[0,1]\times\RR\setminus (B_n\cup [0,1]\times (1,+\infty))$: 
see Figure \ref{fig:Bn}. 

Then we have concatenation expressions 
$$\xi'|_{B_n}=\alpha_1\cdot ...\alpha_l\cdot ...\cdot\alpha_r,$$
$$\xi'|_{B_{n+1}}=\alpha_1\cdot...\cdot\alpha'_l\cdot ...\cdot\alpha_r,$$
where each $\alpha_i$ is a b-path that connects two $b$-points of $\fP$, bounds partially some $D_k$, has image in a subset $U_i$ and $\alpha_l$ and $\alpha'_l$ are the $b$-paths whose union is the boundary of the region $D_{n+1}$ (that is $C(D_{n+1}^0$).  Then $g_i(D_{n+1})$ is contained in a subset $U_j$. Hence $[\alpha_l]$ and $[\alpha'_l]$ are equal in $MD\pi_1^b(U_j,\fP)$. This shows that for every $n$
\begin{equation}\label{eq:xi'}\kappa([\xi'|_{B_n}])=\kappa([\xi'|_{B_{n+1}}]).\end{equation}
Then, we conclude (\ref{eq:xi}) since it is clear that $\kappa([\xi'|_{B_0}])=\kappa([\xi|_{I\times\{1\}}])$ and $\kappa([\xi'|_{B_{m}}])=\kappa([\xi|_{I\times\{0\}}])$ and they are equal by (\ref{eq:xi'}).

\end{proof}

Let $(X,0,d)$ be a metric subanalytic germ and $(\calU,\fP)$ be a finite cover and a set of $b$-points in $X$ satisfying condition $(*)$. We wish to compare $MD\pi_1^b(X,\fP)$ with the colimit $\mathrm{colim} \{MD\pi_1^b(U_{i_1,...,i_r},\fP)\}_{i_1....i_r}$. 
By the universal property that we have just proved for $MD\pi_1^b(X,\fP)$ and the one defining the colimit we obtain a natural morphism of groupoids
\begin{equation}
\label{eq:comparisoncolim}
\sigma:\mathrm{colim} \{MD\pi_1^b(U_{i_1,...,i_r},\fP)\}_{i_1....i_r}\to MD\pi_1^b(X,\fP).
\end{equation}
We wish to find conditions ensuring that $\sigma$ is an isomorphism. 



\begin{theo}
\label{theo:SvKGroupoids}
Let $(X,0,d)$ be a metric subanalytic germ and $(\calU,\fP)$ be a finite cover and a set of $b$-points in $X$ satisfying condition $(*)_b$. Then $\sigma$ is surjective at the level of objects and morphisms. 

Assume  $(\calU,\fP)$ satisfies condition $(*)_b$ and  the following condition  

\begin{itemize}
\item [$(**)_b$] if $\gamma$ is a $b$-path in $(U_i,d_i)$ and  $\gamma'$ is a $b$-path in $(U_j,d_j)$ such that $\gamma\sim_b\gamma'$ in $(X,d)$ then there exists a $b$-path $\delta$ in $U_{ij}$ which is the same $b$-map as $\gamma$ in $U_i$ with the metric $d_i$ and as $\gamma'$ in $U_j$ with the metric $d_j$.
\end{itemize}
Then the morphism $(\ref{eq:comparisoncolim})$ is an isomorphism of groupoids. In other words, if $(\calU,\fP)$ satisfies conditions $(*)_b$ and $(**)_b$ then $MD\pi_1^b(X,\fP)$ is the colimit of the system of groupoids 
$\{MD\pi_1^b(U_{i_1,...,i_r},\fP)\}_{i_1....i_r}$ (this is the usual formulation of Seifert-van Kampen Theorem for groupoids).
\end{theo}

\proof[Proof of Theorem~\ref{theo:SvKGroupoids}]
Since $\calU$ is a subanalytic cover the mapping $\sigma$ is always surjective at the level of objects. 

Let $\gamma$ be a $b$-path in $MD\pi_1^b(X,\fP)$. As in the proof of Theorem~\ref{theo:1stuniversal} we write $[\gamma]$ as 
$$[\gamma]=[\gamma_0\cdot\delta_1^0]\cdot [\overleftarrow{\delta_1^1}\cdot\gamma_1\cdot\delta_2^0]\cdot...\cdot[\overleftarrow{\delta_r^1}\cdot\alpha_r],$$
where each factor is a $b$-path in $MD\pi_1^b(U_i,\fP)$ for a certain $i$. This shows surjectivity at the level of morphisms.

%
%

By theorem~\ref{theo:1stuniversal}, in order to prove the isomorphism of groupoids under condition $(**)_b$ it is enough to show that the groupoid $L:=\mathrm{colim}\{MD\pi_1^b(U_{i_1,...,i_r},\fP)\}_{i_1....i_r}$ and the natural morphisms $b_i:MD\pi_1^b(U_{i},\fP)\to L$ satisfy:
\begin{enumerate}[(i)]
 \item Property $(\dagger)_b$, and 
 \item for any other groupoid $K$ and groupoid morphisms $\kappa_i:MD\pi_1^b(U_i,\fP)\to K$ with the property $(\dagger)_b$ there exists a unique groupoid morphism $\kappa:L\to K$ such that $\kappa\comp b_i=\kappa_i$ for any $i\in I$. 
\end{enumerate}
Both assertions follow from the definition of colimit.

\endproof
Now,we particularize for the  MD fundamental group, similarly to the classical Seifert-van Kampen Theorem:

\begin{cor}\label{cor:svk2abiertos}
Let $\calU=\{U_1,U_2\}_{i\in I}$ be a finite cover of a metric subanalytic germ $(X,0,d)$ by subanalytic open subsets $U_i$ with metrics $d_i$ and $d_j$, both of type (a) or (b). Assume that $U_1$, $U_2$ and $U_1\cap U_2$ are $b$-connected. Choose a $b$-point $\p$ in $U_1\cap U_2$. If $\calU$ satisfies property $(**)_b$ then $MD\pi_1^b(X,\p)$ is isomorphic to the amalgamated product  
$$MD\pi_1^b(U_1,\p) *_{MD\pi_1^b(U_1\cap U_2,\p)} MD\pi_1^b(U_2,\p).$$
\end{cor}
\proof
We have to check that Property $(*)_b$ holds. Let $\p_1$, $\p_2$ be $b$-points in $U_1$ and $U_2$ respectively. If $\p_1\sim_b \p_2$ then by property $(**)_b$ there is a point $\q$ in $U_1\cap U_2$ which is $b$-equivalent to both. Choose a path $\gamma$ connecting $\q$ with $\p$ (use that $U_1\cap U_2$ is $b$-path connected). We can choose $\gamma$ for both required paths in condition $(*)_b$ and choose the homotopies to be constant and equal to $\gamma$. 
\endproof

\begin{cor}
\label{cor:svknerve}
Let $(X,0,d)$ be a metric subanalytic germ. Fix $b$ and let $\calU$ be a cover satisfying Condition $(**)_b$. Assume that each intersection of up to $3$ open subsets of the cover is $b$-connected and has trivial MD $b$-fundamental group. Then the MD $b$-fundamental group of $X$ is isomorphic to the topological fundamental group of the nerve of the cover.
\end{cor}
\proof
Choose a set $\fP$ of base points having one point at each intersection of $1$, $2$ or $3$ open subsets of the cover. By the $b$-connectivity assumptions condition $(*)_b$ is trivially satisfied. Then by Theorem~\ref{theo:SvKGroupoids} $MD\pi_1^b(X,0,d,\fP)$ is the coproduct of the system of fundamental groupoids associated to the cover. From this point the proof is reduced to the same statement in the topological category (see \cite{Hatcher} Proposition 4.G2).

\section{Comparison theorems}\label{sec:Compar}

\subsection{The $\infty$-MD homotopy sets and the homotopy of the link}

\begin{prop}\label{prop:MD homotopy at infinity}
Let $(X, Y, x_0, d, \p)$ be a pointed pair of metric subanalytic germs. Fix $\epsilon > 0$ small enough so that $(L_{X, \epsilon},L_{Y,\epsilon},\p_\epsilon):=(X,Y,\mathrm{Im}(\p)\cap \SSS_\epsilon)$ is the link of the pair. There is a bijection
$$
\zeta : \pi_n(L_{X, \epsilon}, L_{Y,\epsilon},\p_\epsilon)\xrightarrow{\sim} MD\pi_n^\infty(X,Y, \p)
$$
where $\pi_n(L_{X, \epsilon}, L_{Y,\epsilon},\p_\epsilon)$ denotes the standard $n$-th homotopy set. The bijection is a group isomorphism whenever a group structure is defined.
\end{prop}
\begin{proof}
We can assume that $\p$ is a normal point by independence of the base point (see Proposition \ref{prop:indbasepoint}). Then $\p_\epsilon=\p(\epsilon)$.

It is well known that 
 $\pi_n(L_{X, \epsilon}, L_{Y,\epsilon},\p_\epsilon)$ is the quotient of the set of subanalytic $n$-loops in $(L_{X, \epsilon}, L_{Y,\epsilon},\p_\epsilon)$ by subanalytic homotopies (this can be done using a subanalytic triangulation and the simplicial approximation theorem). Now to define $\zeta$, we choose  $h: (C(L_{X, \epsilon}, L_{Y, \epsilon},\p_\epsilon),0)\xrightarrow{\sim} (X,Y, \text{Im}(\p), x_0) $ a subanalytic homeomorphism germ that gives conical structure (see Definition \ref{def:con_struc} and Remark~\ref{rem:lvareduction}) to $(X,0)$ compatible with the subgerm $Im(\p)$ and such that $||h(x,t)||=t$. 
Then, given a subanalytic $n$-loop $\psi:I^n\to (L_{X, \epsilon}, L_{Y,\epsilon},\p_\epsilon)$ we define $\zeta(\psi):C(I^n)\to X$ by the formula $\zeta(\psi)(x,t)=h(\psi(x),t)$. The assignment is obviously well defined. 

Let us show surjectivity. Let $\varphi : C(I^n) \to X$ be any $\infty$-MD $n$-loop in $MD\pi_n^\infty (X,Y,\p)$. Let us denote  $h^{-1}(x) = (\alpha(x), \tau(x))$, where $h$ is the homeomorphism we chose above and we are taking coordinates $(y,t)$ in $C(L_{X, \epsilon})$ as in Notation \ref{not:cone}.
We see first that $\varphi$ is $b$-homotopic to $ (\infty,b)$-loop $\varphi'$ satisfying $\tau'(\varphi'(y,t))=t$. Let $r_y(t)$ be the inverse map germ of the map germ $\tau(\varphi(y,\_))$ defined on $(0,\epsilon)$. Then it is clear that  $\tau(\varphi(y,r_y(t)))=t$. We can take $\varphi'(y,t):=\varphi(y,r_y(t))$ since it is clear that $H_1:C(I^{n+1})\to (X,Y)$ defined by $H_1(y,s,t):=\varphi(y,(1-s)t+sr_y(t))$ defines a homotopy from $\varphi$ to $\varphi'$. 

Now we see that $\varphi'$ is  homotopic to $\varphi''(y,t):=h( \alpha\circ\varphi'(y,\epsilon),t)$ which is obviously in the image of $\zeta$. A homotopy from $\varphi''$ to $\varphi'$ is $H(y,s,t):=h(\alpha(\varphi'(y,\rho(t,s)),t)$ with $t\in (0,\epsilon)$, where $\rho(t,s):=\epsilon$ if $t\geq (1-s)\epsilon$ and $\rho(t,s):=t+(1-s)\epsilon$ if $t\leq (1-s)\epsilon$.

Injectivity is proven applying the same procedure to the homotopies.
\end{proof}

\subsection{Comparison of germs with the outer metric and their horn neighbourhoods}
\label{sec:hurew}

\begin{definition}
\label{def:bclosechains}
Let $d_e$ denote the euclidean metric in $\RR^n$. Let $(X,Y,d_e,O,\p)$ be a pointed pair metric of subanalytic germs embedded in $\RR^n$ with the outer metric $d_e$. For $b<\infty$ we denote by $\calN_b MD\Gamma^b_k(X,Y,\p)$ the set of weak $b$-maps $\varphi:C(I^k)\to (\RR^n,O)$ satisfying
\begin{enumerate}[(a)]
 \item $\lim_{t\to 0}\frac{max\{d_e(\varphi(ty,t),X):y\in I^k\}}{t^{b}}=0.$
 \item for any point $\q$ in $C(\partial I^k)$, the point $\varphi \comp \q$ is $b$-equivalent to a point $\q':(0,\epsilon)\to Y$.
 \item for any normal point $\q$ in $C(\partial I^k\setminus (I^{k-1}\times\{1\}))$, the point $\varphi \comp q$ is $b$-equivalent to $\p$.
\end{enumerate}
We denote by $\calN_b MD\pi^b_k(X,Y,\p)$ the quotient of $\calN MD\Gamma^b_k(X,Y,\p)$ by weak $b$-homotopies $H:C(I^{k+1})\to I$, relative to $C(\partial I^k\setminus (I^{k-1}\times\{1\}))$, preserving the inclusion of $\partial I^{n}$ into $Y$ and satisfying 
$$\lim_{t\to 0}\frac{max\{d_e(H(ty,t),X):y\in I^{k+1}\}}{t^b}=0.$$
\end{definition}

The sets $\calN_b MD\pi^b_k(X,Y,\p)$ have a group structure whose product is the concatenation of $b$-loops whenever this concatenation is possible. 

\begin{prop}
\label{prop:rhoiso}
With the notations of the previous definition there is a bijection (a  group isomorphism when the group structure is defined)
$$\rho:\calN_{b}MD\pi^b_k(X,Y,\p)\to MD\pi^{b}_k (X,Y,\p)$$
for any $k$ and for any $b$.
\end{prop}
\proof
As a consequence of Proposition~\ref{prop:smallqis} we can replace $X$ by its closure without loosing generality. So we assume that $X$ is closed. 

Apply Lemma~\ref{lem:retracciondiscontinua} to the pair $B\supset X$. Let $\calS:=\{S_i\}_{i\in I}$ and $\{f_i\}_{i\in I}$ be the partition and the subanalytic maps predicted in the Lemma \ref{lem:retracciondiscontinua} . Then $\{\overline{S}_i\}_{i\in I}$ is a closed cover of $B$. Let $U$ be the union of the interiors of the sets $S_i$. Then $U$ is a dense subanalytic subset of $B$. A straightforward adaptation of Proposition~\ref{prop:smallqis} shows the bijection
$$\iota:\calN_b MD\pi^{b, U}_k(X,Y,\p)\to \calN_b MD\pi^{b}_k(X,Y,\p),$$
where $\calN_b  MD\pi^{b, U}_k(X,Y,\p)$ is defined in analogy with Definition~\ref{def:smallchain}. 

Given any loop $\varphi\in \calN_b MD\Gamma^{b, U}_k(X,Y,\p)$ we choose a representative $\{(C_j,g_j)\}_{j\in J}$ such that $g(C_j)$ is contained in $S_{i(j)}$ for a certain $i(j)$. Define $\rho'(\sigma):=\{(C_j,f_{i(j)}\comp g_j)\}_{j\in J}$. A straightforward application of the triangle inequality shows that $\rho'(\sigma)$ is a well defined element in $MD \Gamma^{b}_k(X,Y,\p)$. Applying the same procedure to the homotopies we get a well defined mapping 
$$\rho':\calN_b MD\pi^{b,U}_k(X,Y,\p)\to MD\pi^{b}_k(X,Y,\p).$$
Define $\rho:=\rho'\comp\iota^{-1}$. The mapping 
$$MD\pi^{b}_k(X,Y,\p)\to \calN_b MD\pi^{b}_k(X,Y,\p)$$
induced by the inclusion is clearly a left-inverse of $\rho$. It is also a right inverse because given $\varphi\in \calN_b MD\Gamma^{b, U}_k(X,Y,\p)$, the weak $b$-maps $\varphi$ and $\rho'\comp\varphi$ are $b$ equivalent due to Condition~(a) of Definition~\ref{def:bclosechains}.
\endproof

For the next lemma recall the definition of horn neighborhood from~\cite{MDH}:

\begin{definition}[Horn Neighborhood]
\label{def:sphhorn}
Let $X$ be a subanalytic germ embedded in $\RR^n$. We assume the vertex of the cone to be the origin in $\RR^n$. Let $b\in\RR^{+}$. The {\em  $b$-horn neighborhood} of amplitude $\eta$ of $X$ in $\RR^n$ is the union 
$$\calH_{b,\eta}(X):=\bigcup_{x\in X}B(x,\eta \|x\|^b).$$

\end{definition}

\begin{lema}
\label{lem:directlimit2}
With the same notations as above, for any fixed $\eta>0$, the set $\calN_{b}MD\pi^{b}_k(X,Y,\p)$ is in a bijection with the direct limit of sets 
$$\lim_{\stackrel{\longrightarrow}{b'>b}} MD\pi^{b'}_k(\calH_{b',\eta}(X),\calH_{b',\eta}(Y),\p,d_{e}).$$
\end{lema}
\proof
Since we have the inclusion $\calH_{b_1,\eta}(X)\subset \calH_{b_2,\eta}(X)$ if $b_1>b_2$ we have natural maps
$MD\pi^{b_1}_k(\calH_{b_1,\eta}(X),\calH_{b_1,\eta}(Y),\p)\to MD\pi^{b_2}_k(\calH_{b_2,\eta}(X),\calH_{b_2,\eta}(Y),\p)$ defined to be the composition of 
$$MD\pi^{b_1}_k(\calH_{b_1,\eta}(X),\calH_{b_1,\eta}(Y),\p)\to MD\pi^{b_1}_k(\calH_{b_2,\eta}(X),\calH_{b_2,\eta}(Y),\p)\to MD\pi^{b_2}_k(\calH_{b_2,\eta}(X),\calH_{b_2,\eta}(Y),\p).$$
This forms the direct system. 

For any $b'>b$ any $n$-loop representing an element in $MD\pi^{b'}_k(\calH_{b',\eta}(X),\calH_{b',\eta}(Y),\p)$, and the $b'$-homotopies connecting them satisfy property $(a)$ of Definition~\ref{def:bclosechains}. So we have a homomorphism
$$\xi:\lim_{\stackrel{\longrightarrow}{b'>b}} MD\pi^{b'}_k(\calH_{b',\eta}(X),\calH_{b',\eta}(Y),\p)\to \calN_{b}MD\pi^{b}_k(X,Y,\p).$$

For any loop $\varphi=\{(C_j,g_j)\}_{j\in J}$ representing an element in $\calN_{b}MD\pi^{b}_k(X,Y,\p)$ we expand 
$$\max_{j\in J}\{\max\{d_e(g_j(ty,t),X):(y,t)\in C_j\}\}=a_1t^{b'_1}+\theta_1(t),$$ 
$$\max_{j,j'\in J}\{\max\{d_e(g_j(ty,t),g_{j'}(ty,t)):(y,t)\in C_j\cap C_{j'}\}\}=a_2t^{b'_2}+\theta_2(t),$$ 
$$\max_{j\in J}\{\max\{d_e(g_j(ty,t),Y):(y,t)\in C_j\cap\partial C(I^k)\}\}=a_3t^{b'_3}+\theta_3(t),$$ 
where $\theta_i(t)$ vanishes at $0$ to order higher than $b'_i$, and $b'_i>b$ for $i=1,2,3$. Then $\varphi$ represents an element in $MD\pi^{b''}_k(\calH_{b'',\eta}(X),Y,\p)$ for any $b''$ satisfying $b<b''<\min\{b'_1,b'_2,b'_3\}$. This shows the surjectivity of $\xi$. For the injectivity we apply the same argument to the homotopies.
\endproof

\begin{prop} 
\label{prop:directlimit3}
With the same notations as above, there is a bijection
$$\lim_{\stackrel{\longrightarrow}{b'>b}} MD\pi^{\infty}_k(\calH_{b',\eta}(X),\calH_{b',\eta}(Y),\p,d_{e})\cong \lim_{\stackrel{\longrightarrow}{b'>b}} MD\pi^{b'}_k(\calH_{b',\eta}(X),\calH_{b',\eta}(Y),\p,d_{e}).$$
\end{prop}
The main part of the proof consists in given a weak $b'$-loop on the right hand side limit, finding a $b'$-equivalent continuous one on the left hand side. We perform an interpolation trick for that. We need some preliminary work.

\begin{lema}
\label{lem:puntos}
Consider $b'>1$.
Let $\p_1,...,\p_l:(0,\epsilon)\to\dot \calH_{b',\eta}(X)$ be $b'$-equivalent points. 
A point of the form $\sum_i \lambda_i\p_i$ with $\sum_i\lambda_i=1$ is contained in $\dot \calH_{b'',\eta}(X)$ for every $b''<b$ and it is $b'$-equivalent to the $\p_i$.
\end{lema}

\proof
Since the points are $b$ equivalent there exists a positive $A$ such that $||p_i(t)||=At+\theta_i(t)$, with $\lim_{t\to0} \theta_i(t)/t^{b'}=0$ for all $i$, and fixed $\lambda_i$ such that $\sum_i\lambda_i=1$ we have $||\sum_i\lambda_i p_i(t)||=At+\theta(t)$ with $\lim_{t\to0} \theta(t)/t^{b'}=0$.

$$d(X,\sum_{i=1}^l\lambda_i\p_i(t))\leq d(X,\p_1(t))+d(\p_1(t),\sum_{i=1}^l\lambda_i\p_i(t))\leq$$
$$\eta||\p_1(x)||^b+\sum\lambda_i d(\p_1(t),\p_i(t))\leq \eta||\sum\lambda_i\p_i(x)||^b+||\rho(t)||+\sum\lambda_i d(\p_1(t),\p_i(t),$$
where $\lim_{t\to0} \rho(t)/t^{b'}=0$.

The above inequality implies that $\sum_i \lambda_i\p_i$ is contained in $\dot \calH_{b{''},\eta}(X)$ for every $b''<b'$.
\endproof

\textsc{Construction: the skeleton thickening decomposition}. 

For a set of points $A$ we denote by $[A]$ its convex hull. A $n$-simplex in $\RR^N$ is the convex hull of $n+1$ affinely independent points. More generally a polytope is the convex hull of finitely many points. A piecewise linear simplicial complex $|K|$ is a locally finite union of simplexes such that any finite intersection of simplexes is a simplex. A simplex is maximal if it is not contained in any other different simplex. Now we decompose $|K|$ as a union of subsets which refines the decomposition of $|K|$ into maximal simplexes.  


We start defining the decomposition for a single simplex. Let $T$ be the set of vertices of a $k$ dimensional simplex. For any subsets $f\subseteq f'\subseteq T$ we define


$$b_f\ \text{the\ baricenter\ of}\ [f]$$


$$ v^f:=\frac{1}{2}b_f+\frac{1}{2}v \ \text{\ \ for\ any }\ v\in f$$

$$f^{f'}:=\{v^{f'}:v\in f\}\ \text{for\  any }\  f\subseteq f'\subseteq T$$

$$f^{*f'}:= \bigcup_{f'':f\subseteq f''\subseteq f'}f^{f''}$$

\begin{figure}

\includegraphics[scale=0.15]{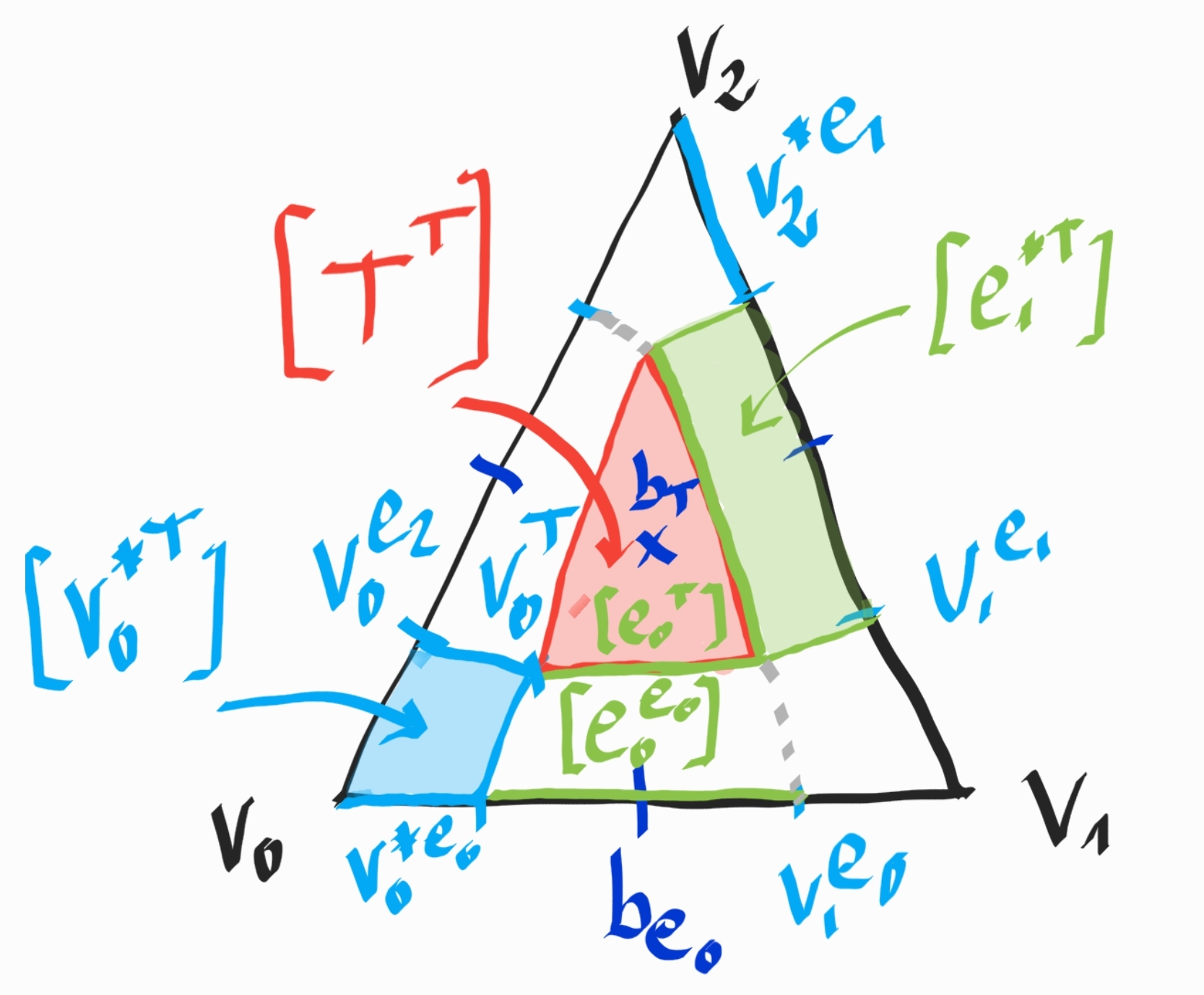}
\caption{Partial decomposition of the simplex $T=\{v_0,v_1,v_2\}$ where the edges are $e_0=[v_0,v_1 ]$, $e_1=[v_1,v_2]$ and $e_2=[v_2,v_0]$.}\label{fig:descomp}
\end{figure}

Note that $[f^{f'}]$ is a simplex of the same dimension as $[f]$,  strictly contained in the interior of $[f']$ and with faces parallel to the ones of $[f]$ (see Figure \ref{fig:descomp}). In particular $[T^T]$ is a simplex of dimension $k$ completely contained in $[T]$. 

Observe that, given $f\subset f'$, 
the sets $[f^{*f'}]$ are of the same dimension than $[f']$. There are $2^{s'-s}$ possible subsets $f''$ with  $f\subseteq f''\subseteq f'$ if $s'=dim([f'])$ and $s=dim([f])$.

We have a decomposition
\begin{equation}\label{eq:decomp}[T]=[T^T] \cup \bigcup_{f:f\subsetneq T} [f^{*T}]
\end{equation}
where 
$f$ runs over the set of faces of $T$ of dimension less than $k$. It can be checked that the sets in the decomposition are polytopes of dimension $k$ and only intersect along polytopes (faces) of smaller dimension. 

Let us also see that every $[f^{*T}]$ is fibred over $[f^{f}]$ with fiber a cube $[0,1]^{k-s}$ with  $s=dim([f])$. Consider in $[T]$ the pushforward metric of the euclidean metric in the standard $k$-simplex $[(1,...,0),....,(0,...,1)]\subset\RR^{k+1}$ by the mapping given by the barycentric coordinates in $[T]$. Consider the orthogonal projection $\pi$ from $[T]$ to $[f]$. It can be easily checked by elementary geometry, that the restriction of $\pi$ to $[f^{T}]$ is a bijection onto $[f^f]$, and that $\pi$ restricted to $[f^{*T}]$ is a trivial bundle with base $[f^f]$ (or $[f^{T}]$) and fiber a cube of the corresponding dimension. So, we have orthogonal trivializations 
$$\Gamma_{f,T}:  [f^{*T}]\to [f^{T}]\times [0,1]^{k-s}$$
where the first projection to $[f^{T}]$ coincides with the orthogonal projection in $[T]$ to $[f]$ with respect to the mentioned metric. Note that by construction if $f\subseteq f'\subseteq T$ and $x\in [f^{*T}]\cap [{f'}^{*T}]$ then the first projections of $\Gamma_{f,T}(x)$ and $\Gamma_{f',T}(x)$ to $[f^{T}]$ and $[{f'}^{T}]$ coincide.

We call the decomposition~(\ref{eq:decomp}) the {\em skeleton thickening decomposition} of $[T]$, and the system of mappings $\Gamma_{f,T}$ is the {\em associated system of trivializations}.

Let $[T]$ be a $k$-simplex and let $[T']$ be a $l$-dimensional face of $[T]$. The intersection of the pieces of the skeleton thickening decomposition of $[T]$ with $[T']$ gives induces the skeleton thickening decomposition of $[T']$. The associated system of trivializations of the skeleton thickening decomposition of $[T']$ coincide with the relevant trivializations of $[T]$, forgetting a $[0,1]^{k-l}$ factor on the right hand side.



Once we have defined the decomposition for a simplex, we define the {\em skeleton thickening decomposition} of $|K|$ decomposing every maximal simplex $[T]$ as in Equation~(\ref{eq:decomp}). We consider the corresponding trivialization mappings $\Gamma_{f,T}$.

Finally, for each symplex $[T]$ of $|K|$ we consider the unique affine homeomorphism
\begin{equation}
\label{eq:dilation}
\tau_T:[T^{T}]\to [T].
\end{equation}

\proof[Proof of Proposition~\ref{prop:directlimit3}]
The natural maps $MD\pi^{\infty}_k(\calH_{b',\eta}(X),\calH_{b',\eta}(Y),\p,d_{e})\to MD\pi^{b'}_k(\calH_{b',\eta}(X),\calH_{b',\eta}(Y),\p,d_{e})$ induce a map
$$\lim_{\stackrel{\longrightarrow}{b'>b}} MD\pi^{\infty}_k(\calH_{b',\eta}(X),\calH_{b',\eta}(Y),\p,d_{e})\to \lim_{\stackrel{\longrightarrow}{b'>b}} MD\pi^{b'}_k(\calH_{b',\eta}(X),\calH_{b',\eta}(Y),\p,d_{e}).$$

We start showing its surjectivity. Consider a $b'$-map $\phi=\{(C_i,g_i)\}_{i\in I}$ representing an element of $MD\pi^{b'}_k(\calH_{b',\eta}(X),\calH_{b',\eta}(Y),\p,d_{e})$. We claim that for any $b<b''<b'$ there exists a continuous subanalytic $k$-loop $\phi':C(I^k)\to \calH_{b'',\eta}(X)$ that is weak $b''$-homotopic to $\phi$ relative to $C(\partial I^k\setminus I^{k-1}\times\{1\})$ by a homotopy preserving the inclusion of $C(\partial I^k)$ into $\calH_{b'',\eta}(Y)$. The claim obviously shows the surjectivity. Let us prove it now. 

By refining we can assume that $C_i\cap C_j$ have empty interior. By the subanalytic triangulation theorem and Remark~\ref{rem:lvareduction} there exists a piecewise linear simplicial complex $|K|$ such that $|K|=I^k$ and 
a subanalytic homeomorphism $h:C(|K|)\to C(I^k)$ of the form $h(x,t)=(h_1(x,t),t)$ in coordinates of $C(I^k)$, that gives a subanalytic triangulation on $C(I^k)$ refining $\{C_i\}$, the natural strata of $I^k$ and $Im(\p)$. 

We will prove the claim for $\phi\comp h$ instead of $\phi$. The proof for $\phi$ consists in composing the obtained continuous $k$-loop and homotopy with $h^{-1}$. So, from now on we abuse notation and identify $\phi$ with $\phi\comp h$, and assume that the decomposition of $I^k$ given by the $C_i$'s coincides with the cone of the piecewise linear triangulation of $I^k$ induced by $|K|$. 




Consider the skeleton thickening decomposition of $|K|$, together with the trivialization mappings associated with the maximal simplexes and their faces. 

Since we are assuming that the decomposition $C(I^k)=\cup_{i\in I}C_i$ is the cone of the decomposition of $|K|$ into maximal simplexes, we have that the maximal simplexes are indexed by $I$. Given any $i\in I$ we define 
\begin{equation}\label{eq:Ri}R_i:=[T_i]\cup \bigcup_{f\subset T_i}\bigcup_{T_j\supset f}[f^{*T_j}].\end{equation}
 
We define a a weak $b'$-map $(C(R_i),\overline{g}_i)$ of $\phi$ as follows:
\begin{equation}\label{eq:gi}  
\overline{g}_i(x,t) = 
     \begin{cases}
       g_i(\tau_{T_i}(x),t) &\quad\text{if } x\in [T_i^{T_i}]\ \ \ \ \  \\
       
    g_i(\tau_{T_i}(y),t)&\quad\text{if} \ x\in [f^{*T_i}] \ \text{and}\   \Gamma_{f,T_i}(x)=(y,u)\\
    g_i(\tau_{T_i}(\overline{y}),t)&\quad\text{if}\  x\in [f^{*T_j}] \ \text{and}\    \Gamma_{f,T_j}(x)=(y,u)\\
     \end{cases}
\end{equation}
where  $\overline{y}\in [f^{T_i}]$ is the image of $y\in[f^{T_j}]$ by the compositions of the indentifications of $[f^{T_j}]$ with $[f^f]$ and $[f^f]$ with $[f^{T_i}]$ via the orthogonal projections in $[T_j]$ and $[T_i]$ respectively. It is obvious that $\overline{g}_i$ is continuous. 

Notice that for every $s$-dimensional face $f\in [T]$ and for every $j$ such that $T_j$ contains $f$, we have the product structure 
$$C(\Gamma_{f,T_j}): C([f^{*T_j}])\to C([f^{T_j}])\times [0,1]^{k-s},$$
and that $g_i$ is constant at the fibres of the composition of $C(\Gamma_{f,T_j})$ with the first projection. This, together with the fact that $\{(T_i,g_i)\}_{i\in I}$ is a weak $b$-map implies that $\overline{\phi}$ is a weak $b$-map.

Note that $\{([T_i], \overline{g}_i|_{[T_i]} )\}_i$ is a representative of $\overline{\phi}$, and that $\{([T_i], \overline{g}_i|_{[T_i]} )\}_i$ is clearly $b'$-homotopic to $\phi$. Then, $\overline{\phi}$ is $b'$-homotopic to $\phi$.

Choose a subanalytic partition of unity $\{\rho_i\}$ adapted to $\{\dot{R}_i\}$ in $I^k$. Then, we define $\phi'(x,t)=\sum_i\rho_i(x)\overline{g}_i(x,t)$. This map is continuous since every $\rho_i(x)\overline{g}_i(x)$ is continuous in $I^k$, its image is inside $\calH_{b'',\eta}(X)$ by Lemma~\ref{lem:puntos}, and it is  $b'$-equivalent to $\overline{\phi}$. This shows the claim and proves surjectivity.

Injectivity is proved with a similar argument applied to the homotopies.
\endproof

Now we are ready for the main result in this section:

\begin{theo}
\label{theo:maincomparacion}
Let $(X, Y, d_{e}, x_0,\p)$ be a pointed pair of metric subanalytic subgerms of $\RR^n$ with the outer metric. For $k\geq 1$ in the absolute case and $k\geq 2$ in the relative one the following assertion holds. For any $b<\infty$ there exists a $b_0>b$ such that for any $b'\in (b,b_0)$ there exists a positive $\epsilon_{b'}$ such that we have isomorphisms
$$MDH^b_k(X,Y;\ZZ)\cong H_k(\calH_{b',\eta}(X)\cap B_{\epsilon_{b'}}\setminus\{x_0\},\calH_{b',\eta}(Y)\cap B_{\epsilon_{b'}}\setminus\{x_0\};\ZZ),$$
$$MD\pi^b_k(X,Y,\p)\cong \pi_k(\calH_{b',\eta}(X)\cap B_{\epsilon_{b'}}\setminus\{x_0\},\calH_{b',\eta}(Y)\cap B_{\epsilon_{b'}}\setminus\{x_0\},\p(t_0)),$$
for $t_0$ small enough. 
\end{theo}
\proof
The MD Homology statement has been proved in~\cite{MDH}.

By Proposition~\ref{prop:directlimit3}, Lemma~\ref{lem:directlimit2} and Proposition~\ref{prop:rhoiso} we have the isomorphism 
$$MD\pi^b_k(X,Y,\p)\cong \lim_{\stackrel{\longrightarrow}{b'>b}} MD\pi^{\infty}_k(\calH_{b',\eta}(X),\calH_{b',\eta}(Y),p,d_{e}).$$
By Proposition~\ref{prop:MD homotopy at infinity} and the conical structure theorem for any $b'$ there exists $\epsilon_{b'}$ so that we have 
$$MD\pi^\infty_k(\calH_{b',\eta}(X),\calH_{b',\eta}(Y),\p)\cong\pi_k(\calH_{b',\eta}(X)\cap B_{\epsilon_{b'}}\setminus\{x_0\},\calH_{b',\eta}(Y)\cap B_{\epsilon_{b'}}\setminus\{x_0\},p(t_0)).$$

The rest of the proof is an easy adaptation of the invertible cobordism techniques of the proof of Theorem~107 and Corollary~118 of~\cite{MDH}.
\endproof

\begin{remark}\label{re:interpol}
Note that the skeleton thickening decomposition can be used as in the proof of \ref{prop:directlimit3} to solve the following convex interpolation problem in a simplicial complex: \emph{Let $K$ be a finite simplicial complex of any dimension.  Assume you have continuous functions  $\{g_{T_i}\}_i$, with values in a convex set $A$, where each $g_{T_i}$ is defined over the simplex $[T_i^{T_i}]$ (see Figure \ref{fig:descomp} and the skeleton thickening decomposition) with $T_i$ a  maximal dimensional simplex of $K$. Find a continuous function defined over the whole $|K|$ that extends the family of functions $\{g_{T_i}\}_{T_i}$. }

%

\end{remark}

\subsection{Finiteness properties for germs with the inner or outer metric and comparision with their tangent cones}\label{sec:inn_out}
After the comparison theorem is established, easy adaptations of the corresponding results for the MD-homology in \cite{MDH} yield the following set of results:
\begin{cor}
\label{cor:finitudes}
Let $(X, Y, d,x_0,\p)$ be a pointed pair of metric subanalytic germs with the inner or the outer metric. Assume $k\geq 1$ for absolute homotopy groups and $k\geq 2$ for relative ones. The following assertions hold
\begin{enumerate}[(i)]
 \item The groups $MD\pi^b_k(X,Y,\p)$ are finitely presented.
 \item There is a finite set of rational numbers $\{b_1,...,b_n\}$ such that if $[b,b']$ does not intersect it then the natural homomorphism
$$MD\pi^{b'}_k(X,Y,\p)\to MD \pi^{b}_k(X,Y,\p)$$
for any $k$ as above.
\item The group $MD\pi^1_k(X,d_{out},\p)$ (for the outer metric) is isomorphic to the $k$-th homotopy group of the punctured tangent cone of $X$.
\item The group $MD\pi^1_k(X,d_{inn},\p)$ (for the inner metric) is isomorphic to the $k$-th homotopy group of the punctured Gromov tangent cone of $X$ (see~\cite{BerLyt07} for the definition).
\end{enumerate}
\end{cor}

\subsection{Metric Hurewicz isomorphism theorem}

After Theorem \ref{theo:maincomparacion} we can prove now Theorem ~\ref{theo:hur}:

\begin{proof}[Proof of Theorem~\ref{theo:hur}]
For the outer metric, given the previous theorem, the proof gets reduced to the routine checking that the comparision isomorphisms are compatible with the Hurewicz homomorphisms in each category. For the inner metric, by~\cite{BirbrairMost:2000}, there exists a different embedding of the germ such that $X$ is Lipschitz normally embedded for the new embedding, and the metric is Lipschitz equivalent with the original inner metric of $X$.
\end{proof}

\subsection{MD homotopy groups of $b$-cones}

We recall that given a bounded subanalytic set $L\subset \RR^k$ and $b\in\cap (0,+\infty )$ we call the {\em outer $b$-cone $(C_{L,out}^{b},d_{out})$ over} $L$ to the germ at the origin of the set
$$
C_{L,out}^{b}=\{(xt^b,t)\in \RR^k\times \RR;\, x\in L\mbox{ and }t\in [0,+\infty )\}
$$
with the outer metric $d_{out}$. Note that the link of $(C_{L,out}^b,d_{out})$ with the induced metric is bilispchitz to $(L,d_{out})$. When $b=1$, we have denoted $C_{L,out}^{1}$ simply by $C(L)$ and call it the \emph{straight cone over $L$}. 

If $b\in\QQ$ the outer $b$-cone is subanalytic, and if $b\notin\QQ$ it is at least definable in the $O$-minimal structure $\RR_{an}^{\RR}$. 

It is also convenient to have a definition of {\em inner metric $b$-cones}. Let $L\subset \RR^k$ be a subanalytic set. 
The inner metric $b$-cone of $L$ will be a straight cone together with a metric such that its link with the induced metric is bilipschitz equivalent to $(L,d_{inn})$. If $L$ is Lipschitz normally embedded (that is, the inner and outer metrics are bi-Lipschitz equivalent), then the inner $b$-cone is by definition $(C_{L,inn}^{b},d_{inn}):=(C_{L,out}^{b},d_{out})$. If $L$ is not Lipschitz normally embedded, by~\cite{BirbrairMost:2000} there exists a different subanalytic embedding $L\approx L'\hookrightarrow\RR^{k'}$ such that $L'$ is Lipschitz normally embedded, and the metric of $(L',d_{out})$ is Lipschitz equivalent with $(L,d_{inn})$, the inner metric that $L$ inherits from its original embedding in $\RR^k$. Then we define $(C_{L,inn}^{b},d_{inn}):=(C_{L',out}^b,d_{out})$. The following remark shows the independence on the embedding and gives an intrinsic description of the inner cone.



\begin{remark}
\label{rem:metricinnercone}
Let $g$ be the Riemannian metric in the smooth part of $L$ inducing the inner metric in $L$. Then, up to bi-Lipschitz equivalence, the Riemannian metric $(dt)^2+t^bg$ induces the metric in $C_L^{b,inn}$.   
\end{remark}

\begin{prop}\label{prop:b-conefundgroup}
Given a bounded connected subanalytic set $L\subset \RR^k$ we have for any $k\geq 1$
\begin{enumerate}[a)]
\item\label{item:b-cone big parameter} If $b' < b$, then $MD\pi_{k}^{b'}(C_{L, out}^{b}, \p)$ and $MD\pi_{k}^{b'}(C_{L, inn}^{b}, \p)$ are trivial.
\item\label{item:b-cone small parameter}  If $b'\geq b$ then we have isomorphisms
$$
MD\pi_{k}^{b'}(C_{L, out}^{b}, \p)\cong MD\pi_{k}^{b'}(C_{L,inn}^{b}, \p)\cong  \pi_{k}(L,\p_0)
$$
\end{enumerate}
\end{prop}

\proof
By definition of the inner cone it is enough to work for the outer cone. In that case the result is a corollary of Theorem~\ref{theo:maincomparacion}.
\endproof

\section{The MD-fundamental group and MD-Homology of normal surface singularities with the inner metric}\label{sec:surfaces}

Let $(X,0,d_{inn})$ be a complex normal surface singularity with its inner metric. Its Lipschitz geometry is completely described in~\cite{BirbrairNeumannPichon:2014}. Before we summarize it we recall the definition of some special metric subanalytic germs for convenience of the reader. They can be read in  Section 11 in \cite{BirbrairNeumannPichon:2014}.  

\begin{definition}[\cite{BirbrairNeumannPichon:2014}]
Choose $1\leq q<q'$. Let $D_2\supset D_{1}$ be the concentric discs centered at the origin of $\RR^2$ of radii $1$ and $2$ and denote by $\dot{D}_1$ the open disc. Consider the subset of $\RR^3$ given by $C^q(D_2)\setminus C^{q'}(\dot{D}_1)$  with the inner metric and denote by $g$ its riemanninan tensor. 
 The germ ${A}^*(q,q')$ is equal to $\SSS^1\times  ((C^q(D_2)\setminus C^{q'}(\dot{D}_1))\setminus \{O\})$ with the riemannian metric $g+t^2d\theta^2$ where $\theta$ parametrizes the factor $\SSS^1$  and $t$ is a l.v.a. parameter of a conical structure on $(C^q(D_2)\setminus C^{q'}(D_1),O)$. 
 We denote by ${A}(q,q')$ the completion of ${A}^*(q,q')$, which adds a point that we call the vertex; ${A}(q,q')$ is homeomorphic to the cone over the annulus $D_2\setminus \dot{D}_1$. The metric extends to the vertex. 
\end{definition}

\begin{definition}
Let $q\in \RR^+$ and let  $\phi:F\to F$ be an orientation-preserving subanalytic diffeomorphism of a compact oriented surface $F$ in $\RR^n$. Consider the inner $q$-cone $C_{F,inn}^q(F)$ of $F$ and let $g$ be its riemannian tensor.  Consider the extension $\overline{\phi}:C_{F,inn}^q(F)\to C_{F,inn}^{q}(F)$
given by $\overline{\phi}(xt^q,t)=(\phi(x)t^q,t)$. We denote by ${B}^*(F,\phi,q)$ the mapping torus of $\overline{\phi}|_{C_{F,inn}^{q}(F)\setminus\{O\}}$ with the metric given by 
$g+t^2d\theta^2$ where $\theta$ parametrizes the factor $\SSS^1$  and $t$ is a l.v.a.  parameter of a conical structure on $C_{F,inn}^q(F)$. 
 We denote by ${B}(F,\phi,q)$ the completion of ${B}^*(F,\phi,q)$, which adds a point that we call the vertex; ${B}(F,\phi,q)$ is homeomorphic to the cone over the mapping torus of $\phi$. The metric extends to the vertex. 
Moreover there is a locally trivial fibration to the puncture disc given by the coordinates $(t,\theta)$.
\end{definition}

Due to \cite{BirbrairNeumannPichon:2014} the inner Lipschitz geometry of $(X,0,d_{inn})$ is described in the following terms: there is a finite number of rational numbers $1=q_1<q_2<...<q_n$ and a canonical subanalytic decomposition in closed subspaces
\begin{equation}
\label{eq:ratedecomp}
X=\bigcup_{k=1}^n Y_k\cup\bigcup_{1\leq i<j\leq n}Z_{i,j}
\end{equation}
that endowed with the inner metric have the following properties
\begin{enumerate}
 \item $Y_1$ is metrically conical, that is bi-Lipschitz subanalytic homeomorphic to the $1$-cone over its link with the inner metric.
 \item There is a subanalytic map $\xi:\overline{X\setminus Y_1}\to D_\epsilon^*$ which is a locally trivial fibration, where $D^*_\epsilon$ denotes the punctured disc of a certain radius $\epsilon$. The map restricts to a locally trivial fibration over each piece of the decomposition of $X\setminus Y_1$.
 \item For each $i>1$ denote by $\phi_i:F_i\to F_i$ the monodromy of $\xi:(Y_i)_\epsilon|_{\partial D^*_\epsilon}\to \partial D^*_\epsilon$. Then $Y_i$ is bilipschitz subanalytic diffeomorphic to $B(F_i,\phi_i,q_i)$ in a such a way that $\xi$ is compatible way with $\xi$ and the natural projection of $B(F,\phi_i,q_i)$   to its base puctured disc given by the coordinates $(t,\theta)$.
\item Every $Z_{i,j}$ is bilipschitz subanalytic diffeomorphic to a (possibly empty) disjoint union of pieces $A(q_i,q_j)$.
\item For a sufficiently small $\epsilon>0$ the decomposition 
\begin{equation}
\label{eq:linkdecomp}
 X_\epsilon=\bigcup_{k=1}^n (Y_k)_\epsilon\cup\bigcup_{1\leq i<j\leq n}(Z_{i,j})_\epsilon
\end{equation}
is a (non-minimal) JSJ-decomposition of the link $X_\epsilon$ (here $Z_\epsilon$ denotes $Z\cap\SSS_\epsilon$), and the decomposition~(\ref{eq:ratedecomp})
is the cone over this decomposition by a subanalytic conical structure. In~\cite{BirbrairNeumannPichon:2014} it is proved that this JSJ decomposition together with the rates $q_i$ is canonically determined and determines the inner Lipschitz geometry of $(X,0)$. This decomposition does not coincide in general with the minimal JSJ-decomposition.
\item Any continuous path joining two points in adjacent pieces in the decomposition (\ref{eq:linkdecomp}) passes through the common boundary.  
\end{enumerate}

\begin{remark}
Although in~\cite{BirbrairNeumannPichon:2014} the authors do not mention the subanalyticity of the decompostition, it holds by their construction: they start with a subanalytic decomposition of $(X,0)$ built from a carrousel associated with the discriminant of a generic projection (see Section~12 of~\cite{BirbrairNeumannPichon:2014}), and after glueing finitely many pieces together in order to reach the canonical decomposition. Each of the pieces of the initial decomposition are subanalytic bilipschitz homeomorphic to the models described above. The same statement for the canonical decomposition holds since the gluing rules of Lemma 13.1~in~\cite{BirbrairNeumannPichon:2014} preserve subanalyticity because they consists in gluing a finite number of subanalytic sets. 

Notice also that the monodromy of the subanalytic fibration $\xi$ and its restriction to each of the pieces has a subanalytic representative as an application of Hardt's trivialization theorem.

For the MD fundamental group computation below one can  work instead with the initial subanalytic decomposition of $(X,0)$ produced in Section~12 of~\cite{BirbrairNeumannPichon:2014} without any further gluing, obtaining the same result.
\end{remark}

Observe that $Y_k$ meets $Z_{ij}$ if and only if $k=i$ or $k=j$. We start also recalling that for every $k\in \{1,...,n\}$ we have that the boundary $(\partial Y_k)_\epsilon$ is a disjoint union of tori. Each tori is fibred over $\SSS^1$ by $\xi$, with fibre a disjoint union of circles. The boundary $\partial Y_k$ is subanalytically homeomorphic to the cone over $(\partial Y_k)_\epsilon$. With respect to the induced inner metric in $\partial Y_k$, the base $\SSS^1$ of the fibration collapses its metric at rate $1$ and the fibres at rate $q_k$. Choose a collar $(C_k)_\epsilon$ of $(\partial Y_k)_\epsilon$ in $(Y_k)_\epsilon$. This collar is the link of a piece $C_k\subset Y_k$ with the property that $C_k$ is a disjoint union of anular pieces of type $A(q_k,q_k)$. It will be important later that both rates of the annular piece are equal. The piece $C_k$ decomposes as a disjoint union as
$$C_k=C_k^-\coprod C_k^+,$$
where $C_k^-$ contains the boundary components of $Y_k$ that intersect pieces of the decomposition of $X$ collapsing slower (that is $\partial Y_k\cap Z_{k-1,k}$), and $C_k^+$ contains the boundary components of $Y_k$ that intersect pieces collapsing faster (that is $\partial Y_k\cap Z_{k,k+1}$). Notice that $C_1^-$ and $C_n^+$ are empty. We choose the $C_k$ so that $\xi|_{C_k}$ is a locally trivial fibration. 

Define 
$$\widetilde{Z}_{i,j}:=Z_{i,j}\cup C_j^-\cup C_i^+$$
The piece $\widetilde{Z}_{i,j}$ is the union of $Z_{i,j}$ with all the constant collapsing rate annular pieces defined above that intersect it.

For every $b\in [1,+\infty)$ we define

$$U_{>b}:=\bigcup_{q_k>b} Y_k\cup \bigcup_{q_i>b}\widetilde{Z}_{i,j}.$$

$$\widetilde{U}_{>b}:=\bigcup_{q_k>b} Y_k\cup \bigcup_{q_j>b}\widetilde{Z}_{i,j},$$


For any $b\in [1,+\infty)$ we define the cover $\calU^b$ of $X$ that we will use to apply our version of the Seifert van Kampen Theorem for the $MD \pi^b_1$ of Theorem \ref{theo:SvKGroupoids}:
\begin{equation}
\label{eq:calUb}
\calU_b:=\{Y_k: q_k\leq b\}\cup\{\widetilde{Z}_{i,j}: q_i<q_j\leq b\} \cup \{\widetilde{U}_{>b}.\}
\end{equation}  
Note that the intersection of any 3 sets in $\calU_b$ is empty. Whenever the intersection of two is non-empty it is a union of $C_k^{-}$ and $C_k^{+}$ pieces.

We fix a l.v.a. subanalytic conical structure 
\begin{equation}\label{eq:fixh}h:C(X_\epsilon)\to (X,0)
\end{equation}
compatible with the decomposition in (\ref{eq:ratedecomp}) and with the $C_k^{-}$ and $C_k^{+}$ pieces. 
If $Q$ is a set of points in $X_\epsilon$ we consider the points given by retraction lines from points in $Q$ and denote  \begin{equation}\label{eq:notfh}\fh(Q):=\{h|_{C(\{z\})}:z\in Q\}
\end{equation}
where $C(\{z\})$ denotes the subcone of $C(X_\epsilon)$.

For any $b\in [1,+\infty)$ we choose a collection $P^b$ of points  that meets all the connected components of the interior of $(\widetilde{U}_{>b})_\epsilon$, of $(Y_k)_\epsilon$,  $(C_k^-)_\epsilon$ and $(C_k^+)_\epsilon$ for every $k$ with $q_k\leq b$  and  of $(Z_{i,j})_\epsilon$ with $q_i<q_j\leq b$. The conical structure induces a collection l.v.a points $\fP^b:=\fh(P^b)$ in $X$.


\begin{prop}
\label{prop:check***}
For any $b\in [1,+\infty)$ the pair $(\calU^b, \fP^b)$ satisfies conditions $(*)_b$ and $(**)_b$ in Theorem  \ref{theo:SvKGroupoids} considering in every set of $\calU^b$ the inner metric $d_i$ induced by inner metric in $(X,0)$, that is as in case (b) at the beginning of Section \ref{subsec:svk***}. Consequently, the groupoid $MD\pi_1^{b}(X,0,d_{inn},\fP^b)$ is the colimit of the system of groupoids associated to the cover $\calU^b$.
\end{prop}
\begin{proof}
Condition $(*)_b$ follows from $(**)_b$ since the cover only has $2$-fold intersections. We check condition $(**)_b$. Let $C$ be the intersection of two sets $B$ and $B'$ of the decomposition. Then $C$ is either a $C_k^-$ or a $C_k^+$ piece, which we call $C$. The collapsing rate of this annular piece is $q_k<b$. This implies that 
$$d_{inn}((B\setminus C)_\epsilon,(B'\setminus C)_\epsilon)=a\epsilon^{q_k}+\theta(\epsilon)$$
for $a>0$ and $\theta(\epsilon)$ decreasing faster than $\epsilon^{q_k}$. Moreover,  $d_{inn}((B\setminus C)_\epsilon,(X\setminus C)_\epsilon)$ and $d_{inn}((B'\setminus C)_\epsilon,(X\setminus C')_\epsilon)$ are of the same form. 

Let $\gamma:C(I)\to B$ and $\gamma':C(I)\to B'$ be weak $b$-paths which are $b$-equivalent in $X$. Then for any fixed $s\in I$ we have that $d_{inn}(\gamma(s,t),\gamma'(s,t))$ decreases faster than $t^b$ where $b>q_k$. This, together with property $(6)$ of the decomposition and the previous bound implies that either $\gamma(s,t)$ or $\gamma'(s,t)$ belongs to $C$ for $t$ small enough. Therefore we split $I$ as a union $I=I_\gamma\cup I_{\gamma'}$ such that $I_\gamma$ and $I_{\gamma'}$ are disjoint unions of closed intervals and $I_\gamma\cap I_{\gamma'}$ is finite, $\gamma|_{C(I_\gamma)}$ has image in $C$ and $\gamma'|_{C(I_{\gamma'})}$ has image in $C$. The path $\delta:C(I)\to C$ defined by $\delta|_{C(I_\gamma)}:=\gamma|_{C(I_\gamma)}$ and $\delta|_{C(I_{\gamma'})}:=\gamma'|_{C(I_{\gamma'})}$ is $b$-equivalent both to $\gamma$ and $\gamma'$. 
\end{proof}

We study in the following lemmata the $MD\pi^b_1$ of each element of the cover $\calU^b$  and their intersections considering always the inner metric on them induced by inner metric in $(X,0)$, that is as in case (b) at the beginning of Section \ref{subsec:svk***}.

\begin{lema}
\label{lem:Ypiece}
There is an isomorphism of fundamental groupoids
$$MD\pi^b_1(Y_k,\fh(Q))\cong \pi_1((Y_k)_\epsilon,Q) \ \ \text{for \ }b\geq q_k,$$
where $Q$ is any set of points in $(Y_k)_\epsilon$ that meets all the connected components. 
\end{lema}
\proof
For $k=1$ the result follows by  Proposition~\ref{prop:b-conefundgroup} because $Y_1$ is metrically conical. For $k>1$ we consider the fibration 
$$\xi|_{Y_k}:Y_k\to D^*_\epsilon.$$
Restricting over the boundary $\partial D^*_\epsilon$ we have a fibration 
$$\xi|_{(Y_k)_\epsilon}:(Y_k)_\epsilon\to \partial D^*_\epsilon$$
whose fibre is a surface $F_i$. The surface $F_i$ decomposes in connected components as
$$F_i=\coprod_{j=1}^N\coprod_{k=1}^{M_j} F_{i,j,k},$$
where the components $F_{i,j,1},...,F_{i,j,M_j}$ are interchanged cyclically by the monodromy. Two components $F_{i,j,k}$ and $F_{i,j',k'}$ are in the same connected component of $Y_k$ if and only if $j=j'$. In what follows we complete the proof for the special case $N=1$ and $M_1=1$. The general case is exactly the same with some more notational complication, and this special case the Lipschitz geometry appears in a more transparent way.

Consider the decomposition
\begin{equation}
\label{eq:discdecomp}
D^*_\epsilon=D^*_\epsilon\setminus\{(t,0),t>0\}\cup D^*_\epsilon\setminus\{(t,0),t<0\} 
\end{equation}
and pullback this decomposition of $D^*_\epsilon$ by $\xi$ to a decomposition of $Y_k$. Each of the two pieces of the decomposition is connected and admits the $b$-cone over $F_i$ as a metric deformation retract: this means that the $b$-cone of $F'_i$ is included in each of this pieces and that the inclusion is a metric homotopy equivalence in the sense of Definition~\ref{def:1homotopy}. The intersection of the two pieces splits as the disjoint union of two connected components, which are the preimages by $\xi|_{Y_k}$ of the two connected components of $D^*_\epsilon\setminus\{(t,0),t\in\RR\}$. Each of these connected components admits the $b$-cone over $F_i$ as a metric deformation retract. 

Therefore, by Proposition~\ref{prop:homotopyinvariance} and Proposition~\ref{prop:b-conefundgroup} the $b-MD$ of each of the two pieces and of the connected components of the intersection between them  of the decomposition is equal to the topological fundamental group of $F_i$. 

We check hypothesis $(*)_b$ and $(**)_b$ in Theorem~\ref{theo:SvKGroupoids} as in the previous Proposition. Then, applying Theorem~\ref{theo:SvKGroupoids} to the decomposition and  the Seifert van-Kampen Theorem for groupois to the corresponding decomposition of $(Y_k)_\epsilon$, one observe that both colimit computations are the same and the result is proven.
\endproof

\begin{lema}
\label{lem:Zpiece}
There is an isomorphism of fundamental groupoids
$$MD\pi^b_1(\widetilde{Z}_{i,j},\fh(Q))\cong \pi_1((\widetilde{Z}_{i,j})_\epsilon,Q) \ \ \text{for \ }b\geq q_j>q_i, $$
where $Q$ is any set of points in $(\widetilde{Z}_{i,j})_\epsilon$ that meets all the connected components. 
\end{lema}
\proof
The boundary of $(\widetilde{Z}_{i,j})_\epsilon$ is a disjoint union of tori classified in two kinds: those such that the piece of $(\widetilde{Z}_{i,j})_\epsilon$ induced by them by the conical structure collapse at rate $q_i$, and those that the associated conical piece collapse at rate $q_j$. Recall that $q_j>q_i$. Denote by $(T_i)_\epsilon$ the union of tori of $q_i$ type, and by $T_i$ its associated conical piece. The inclusion 
$T_i\hookrightarrow \widetilde{Z}_{i,j}$ is a metric homotopy invariance according with Definition~\ref{def:1homotopy}. This reduces the problem to prove the isomorphism
$$MD\pi^b_1(T_i)\cong \pi_1((T_i)_\epsilon),$$
but this is entirely analogous to the proof of the previous lemma.
\endproof

Exactly the same proof yields:

\begin{lema}
\label{lem:Cpiece}
There are isomorphisms of fundamental groupoids
$$MD\pi^b_1(C_{k}^-,\fh(Q^-))\cong \pi_1((C_{k}^-)_\epsilon,Q^-),$$
$$MD\pi^b_1(C_{k}^+,\fh(Q^+))\cong \pi_1((C_{k}^+)_\epsilon,Q^+) \ \ \text{for \ }b\geq q_k$$
where $Q^*$ is any set of points in $(C_k^*)_\epsilon$ that meets all the connected components. 
\end{lema}
To study the  $b$-MD fundamental groupoid of the piece $\widetilde{U}_{>b}$ we define the space $(\widetilde{V}_{>b})_\epsilon$ as follows. Let us start considering the fibration 
$$\xi|_{(U_{>b})_\epsilon}:(U_{>b})_\epsilon\to \partial D^*_\epsilon$$
with fibre $F_{>b}$. Consider the decomposition in connected components $$F_{>b}=\coprod_{j=1}^N\coprod_{k=1}^{M_j} F_{>b,j,k},$$
where the components $F_{>b,j,1},...,F_{>b,j,M_j}$ are interchanged cyclically by the monodromy. Two components $F_{>b,j,k}$ and $F_{>b,j',k'}$ are in the same connected component of $(U_{>b})_\epsilon$ if and only if $j=j'$. We denote the decomposition of $(U_{>b})_\epsilon$ in connected components by 
$$(U_{>b})_\epsilon=\coprod_{j=1}^N(U_{>b})_\epsilon^j.$$
For each $j$ there exists a fibration 
$$\eta_j:(U_{>b})_\epsilon^j\to\SSS^1$$ 
with connected fibre and a unique covering $\rho_j:\SSS^1\to\partial D^*_\epsilon$ of degree $M_j$ such that $\rho_j\comp\eta_j=\xi|_{[(U_{>b})_\epsilon]_j}$. 
Define
$$(\widetilde{V}_{>b})_\epsilon:=(\widetilde{U}_{>b})_\epsilon\cup_{(U_{>b})_\epsilon} \coprod_{j=1}^N Cyl(\eta_j)$$
where $Cyl(\eta_j)$ is the mapping cylinder of and $\cup _{(U_{>b})_\epsilon}$ denotes the gluing of each piece  $Cyl(\eta_j)$ along $(U_{>b})_\epsilon\cap Cyl(\eta_j)$.

\begin{lema}
\label{lem:Upiece}
There is an isomorphism of fundamental groupoids
$$MD\pi^{b}_1(\widetilde{U}_{>b},\fh(Q))\cong \pi_1((\widetilde{V}_{>b})_\epsilon,(\widetilde{U}_{>b})_\epsilon)$$
where $Q$ is any set of points in $(\widetilde{U}_{>b})_\epsilon$ that meets all the connected components. 
\end{lema}
\proof
In order to compute the left hand side notice that the $\xi|_{\widetilde{U}_{>b}}:\widetilde{U}_{>b}\to D^*_\epsilon$ induces a decomposition of $\widetilde{U}_{>b}$ by pullback of the decomposition~(\ref{eq:discdecomp}). Arguments similar to the proof of Proposition~\ref{prop:check***} show that Seifert- van Kampen Theorem~\ref{theo:SvKGroupoids} can be applied to this decomposition. 

Like in the proof of Lemma~\ref{lem:Ypiece} we complete the proof for the special case $N=1$ and $M_1=1$, since the general case is similar.

Then, as in the previous two lemmata, each of the two pieces of the decomposition, and each of the two connected component of the intersection between them is metrically homotopy equivalent to a connected component of the preimage $B=(\xi|_{\widetilde{U}_{>b}})^{-1}(\{(t,0):t>0\}$ of the ray $\{(t,0):t>0\}$. 
It is clear that $B$ is $b$-contractible.

In order to compute the right hand side observe that, similarly, we have a fibration $\xi_\epsilon:(\widetilde{V}_{>b})_\epsilon\to \partial D^*_\epsilon$, and the topological Seifert-van Kampen Theorem for groupoids can be applied to the decomposition of $(\widetilde{V}_{>b})_\epsilon$ obtained by pullback of the restriction of decomposition ~(\ref{eq:discdecomp}) to $\partial D^*_\epsilon$. Each connected component of any finite intersection of subsets of the decomposition is contractible. 

Comparing the colimit computations for the left and right hand sides we conclude.
\endproof


%

Finally, the computation of the whole $MD \pi^b_1(X)$ can be codified in  the topology of the following space, which we can call it the \emph{$(b,1)$-homotopy model} of $X$:
\begin{equation}
\label{eq:bmodel}
X^{b}_\epsilon:=X_\epsilon\cup_{(U_{>b})_\epsilon} Cyl(\xi|_{(U_{>b})_\epsilon})=
 X_\epsilon\cup_{\widetilde{U}_{>b}}(\widetilde{V}_{>b})_\epsilon.
\end{equation}
The space $X^b_\epsilon$ can be understood as the result of fibrewise identifying to a point the connected components of the part of the fibres of $\xi$ that collapse to a rate higher than $b$. It has the homotopy type of a plumbed $3$-manifold in which several circles are identified (a ``branched $3$-manifold'' in the language of~\cite{BirbrairNeumannPichon:2014}). 

Observe that if $b\geq b'$ we have a natural continuous map 
$$\alpha_{b,b'}:X^b_\epsilon\to X^{b'}_\epsilon.$$

\begin{theo}
\label{theo:MDpi1inner} Let $(X,0,d_{inn})$ be a normal surface singularity wuth the inner metric. 
Let $P$ be a set of points in $X_\epsilon$ that meets the interior of every connected components of $Y_k$, $C_k^+$, $C_k^-$ and  every $Z_{ij}$. Let $\fP:=  \fh(P)$ be defined after (\ref{eq:notfh}). Then $\pi_1(X^{b}_\epsilon, P)$ is isomorphic to $MD\pi_1^b(X, \fP)$ for any $b\geq 1$. 

If $p$ is any point in $X_\epsilon$ and $\p$ any l.v.a. point in $X$ there is an isomorphism of $\BB$-groups from
$$...\to \pi_1(X^{b}_\epsilon, p)\to \pi_1(X^{b'}_\epsilon, p)\to ...$$
to $MD\pi_1^{\star}(X,0,d_{inn},\p)$, that is to $$...\to MD\pi_1^b(X, \p)\to MD\pi_1^{b'}(X, \p)\to ... .$$
\end{theo}
\proof
The proof consists in comparing the computation of $MD\pi_1^b(X, \fP)$ by the $MD$ Seifert-van Kampen Theorem, Theorem \ref{theo:SvKGroupoids} with the coverings $\calU^b$ in (\ref{eq:calUb}), and of $\pi_1(X^{b}_\epsilon, P)$ by the topological Seifert van-Kampen Theorem associated with the covering obtained by restriction of (\ref{eq:calUb}), and use Lemmata~\ref{lem:Ypiece},~\ref{lem:Zpiece},~\ref{lem:Cpiece} and~\ref{lem:Upiece}. 
\endproof

Using the covering (\ref{eq:calUb}) and the models (\ref{eq:bmodel})  we can also compute the MD homology of a complex surface singularity: 

\begin{theo}\label{theo:MDHSurfaces} 
Let $(X,0,d_{inn})$ be a normal surface singularity with the inner metric. For every $n\in \NN$ we have isomorphisms from 
$$...\to H_n(X^{b}_\epsilon)\to H_1(X^{b'}_\epsilon)\to ...$$
to $MDH_n^{\star}(X,0,d_{inn},\p)$, that is to $$...\to MDH_n^b(X)\to MDH_1^{b'}(X)\to ... .$$
\end{theo}
\proof
The proof is similar to the previous Theorem, but we have to replace Seifert- van Kampen arguments by Mayer-Vietoris ones.
We modify slightly the covering: consider
\begin{equation}
\label{eq:calUb*}
\calU^*_b:=\{{Y}^*_k: q_k\leq b\}\cup\{\widetilde{Z}^*_{i,j}: q_i<q_j\leq b\} \cup \{\widetilde{U}^*_{>b}.\}
\end{equation}  
defined as follows. Split the annuli $C_k^{\bullet}$, for $\bullet\in \{+,-\}$ in $3$ equal subannuli
$$C_k^{\bullet}=C_k^{\bullet,Y}\cup C_k^{\bullet,\mathrm{middle}}\cup C_k^{\bullet,Z},$$
where $C_k^{\bullet,Y}$ is the third of the annulus adjacent to the $Y$-piece and $C_k^{\bullet,Z}$ is the piece of the annulus adjacent to the $Z$-piece.
Define
$${Z}_{ij}^*:=Z_{ij}\cup C_j^{-,Z}\cup C_j^{-,\mathrm{middle}}\cup C_i^{+,Z}\cup C_i^{+,\mathrm{middle}},$$
$${Y}^*_k:=Y_k\cup C_j^{-,Y}\cup C_j^{-,\mathrm{middle}}\cup C_i^{+,Y}\cup C_i^{+,\mathrm{middle}},$$
and $\widetilde{U}_{>b}^*$ similarly. 

This is a $b$-covering as in Definition 92 in \cite{MDH}: the subsets extending each of the subsets of the cover and their finite intersections are obtained adding the relevant $1/3$-pieces of the corresponding annuli, and checking that such a choice works follows is a simple use of the MD Homology invariance by metric homotopy. Then, we can apply repeately the Mayer-Vietoris type theorem, Theorem 98 in \cite{MDH} in order to compute the $b-MD$ Homology of $X$. Computing the Homology groups of $X_\epsilon^b$ applying repeatedly the Mayer-Viertoris sequence for ordinary homology for the decomposition of $X_\epsilon^b$ corresponding to the decomposition of $X$ defined above, and comparing the computations yields the result.  
\endproof

Let us finish with a few open problems:
\begin{problem}\label{prob:1}
Let $(X,O)$ be a normal surface singularity,
\begin{enumerate}
 \item If $(X,O)$ is not a cyclic quotient, does the MD-fundamental group determine the inner geometry of a normal  complex surface singularity? This is motivated by the corresponding statement, due to Waldhausen, for the topology.
 \item Find a homotopy model computing $MD\pi_1^b(X,O,d_{out})$.
 \item If the natural homomorphism $MD\pi_1^{b}(X,O,d_{inn})\to MD\pi_1^b(X,O,d_{out})$ is an isomorphism for every $b$, is $(X,O)$ Lipschitz normally embedded?
\end{enumerate}
\end{problem}

\begin{problem}
\label{prob:2}
 Compute $MD\pi_1^b$ and $MD H_\star^b$ for any Brieskorn-Pham singularity. The higher homotopy group computation should be very hard, since it contains the homotopy groups of spheres.
\end{problem}

\end{document}